%% file: affine4321.tex
\documentclass[12pt]{article}

\usepackage{fullpage,amssymb,amsmath,amsthm,enumerate,hyperref,tikz,mathdots}

\pdfoutput=1     %  for arxiv

\everymath{\displaystyle}
\allowdisplaybreaks[3]

\newtheorem{thm}{Theorem}[section]
\newtheorem{cor}[thm]{Corollary}
\newtheorem{lemma}[thm]{Lemma}
\newtheorem{prop}[thm]{Proposition}
\newtheorem{rem}[thm]{Remark}
\newtheorem{defi}[thm]{Definition}

\newtheorem{conj}[thm]{Conjecture}

\newcommand{\Seq}{\textbf{\textup{Seq}}}

\newcommand{\Dom}{\textbf{\textup{Dom}}}
\renewcommand{\Im}{\textbf{\textup{Im}}}

\newcommand{\Z}{\mathbb{Z}}
\newcommand{\R}{\mathbb{R}}
\newcommand{\E}{\widetilde{S}^{/\!/}}

\newcommand{\Av}{S}

\newcommand{\AvBA}{\E}

\newcommand{\gr}[1]{\textup{gr}(#1)}
\newcommand{\upgr}[1]{\overline{\textup{gr}}(#1)}
\newcommand{\logr}[1]{\underline{\textup{gr}}(#1)}

\newcommand{\deck}{(k{+}1)\cdots 1}

\newcommand{\wass}{\textup{Wass}}
\newcommand{\collec}[1]{\langle{#1}\rangle}
\newcommand{\pargram}{\diamondsuit}

\newcounter{i}

\newcommand{\drawpermutation}[3][1]{\begin{tikzpicture}[scale=0.5,baseline=(O.base)]
\setcounter{i}{0}
\foreach \j in {#2} {
\stepcounter{i}
\draw (0.5*#1,\value{i}*#1) -- (#3*#1+0.5*#1,\value{i}*#1);
\draw (\value{i}*#1,0.5*#1) -- (\value{i}*#1,#3*#1+0.5*#1);
\node at (\value{i}*#1, 0) {\footnotesize$\j$};
\draw[fill] (\value{i}*#1, \j*#1) circle (0.2);
}
\node (O) at (#3*0.5*#1,#3*0.5*#1) {};
\end{tikzpicture}}

\newcommand{\drawpattern}[4][1]{\begin{tikzpicture}[scale=0.5,baseline=(O.base)]
\foreach \x in {1,...,#3} {
\draw (0.5*#1,\x*#1) -- (#3*#1+0.5*#1,\x*#1);
\draw (\x*#1,0.5*#1) -- (\x*#1,#3*#1+0.5*#1);
}
\setcounter{i}{0}
\foreach \j in {#2} {
\stepcounter{i}
\node at (\value{i}*#1, 0) {\footnotesize$\j$};
\draw[fill] (\value{i}*#1, \j*#1) circle (0.2);
\foreach \k in {#4} {
\ifnum \j=\k
\draw [thick] (\value{i}*#1, \j*#1) circle (0.4);
\fi
}
}
\node (O) at (#3*0.5*#1,#3*0.5*#1) {};
\end{tikzpicture}}

\newcommand{\neal}[1]{\mbox{}{\sf\color{red}[#1]}\marginpar{\color{red}\Large$*$}}

\begin{document}

\title{Bounded affine permutations\\ 
   II. Avoidance of decreasing patterns}
   %\thanks{N.M. was upported in part by a Discovery Grant from NSERC Canada, and by a Minor Research Grant from the Faculty of Science at York University}

\author{Neal Madras
\footnote{Supported in part by a Discovery Grant
from NSERC Canada, and by a Minor Research Grant from the Faculty of
Science at York University}
\\ Department of Mathematics and Statistics \\
York University \\ 4700 Keele Street  \\ Toronto, Ontario  M3J 1P3 Canada 
\\  {\tt  madras@yorku.ca} \vspace{.5pc}
\\ \vspace{.5pc} and \\
Justin M. Troyka\\
Department of Mathematics and Computer Science \\
Davidson College \\
Davidson, NC 28035
\\ {\tt jutroyka@davidson.edu} }
\maketitle

\begin{abstract}
We continue our study of a new boundedness condition for affine permutations, motivated by the 
fruitful concept of periodic boundary conditions in statistical physics.  
We focus on bounded affine permutations of size $N$ that avoid the 
monotone decreasing pattern of fixed size $m$.  
We prove that the number of such permutations is asymptotically 
equal to $(m-1)^{2N} N^{(m-2)/2}$ times an explicit constant
as $N\to\infty$.
For instance, the number of bounded affine permutations of size $N$ that avoid $321$ is asymptotically equal to $4^N (N/4\pi)^{1/2}$.
We also prove a permuton-like result for the scaling limit 
of random permutations from this class,
showing that the plot of a typical bounded affine permutation avoiding $m\cdots1$ looks like $m-1$ random lines of slope $1$ whose $y$ intercepts sum to $0$. \vspace{.3pc}

\noindent\textbf{MSC classes:} 05A05 (primary), 05A16, 60C05, 60G57 \vspace{.3pc}

\end{abstract}

% INTRO MATERIAL FROM FIRST PAPER

\section{Introduction}
   \label{sec-intro}
This paper is a continuation of the research begun in our companion paper \cite{MT1}. Accordingly, some of the text and figures in this introduction are drawn from \cite[Sec.\ 1]{MT1}.

Pattern-avoiding permutations have been studied actively in the combinatorics literature for the past four decades.  
(See Section \ref{sec-definitions} for definitions of terms we use.)
Some sources on permutation patterns include: \cite{Bevan} for essential terminology, \cite[Ch.\ 4]{BonaCP} for a textbook introduction, and \cite{VatterSurvey} for an in-depth survey of the literature.
Pattern-avoiding permutations arise in a variety of mathematical contexts, particularly algebra and the analysis of algorithms.
Research such as \cite{Crites, BilleyCrites} have extended these investigations by considering
affine permutations that avoid one or more (ordinary) permutations as patterns.

\begin{defi} 
   \label{def.affine}
An \emph{affine permutation} of size $N$ is a bijection $\sigma \colon \Z \to \Z$ such that:
\begin{enumerate}[(i)]
\item $\sigma(i+N) \,=\, \sigma(i)\,+\, N$ for all 
$i \in \Z$, and 
\item $\sum_{i=1}^N \sigma(i) = \sum_{i=1}^N i$.
\end{enumerate}\end{defi}

Condition (\textit{ii}) % can also be written as $\sum_{i=1}^n (\sigma(i) - i) = 0$. It 
can be viewed as a ``centering'' condition, since any
bijection satisfying (\textit{i}) can be made to satisfy 
(\textit{ii}) by adding a constant to the function.
The affine permutations of size $N$ form an infinite Coxeter group under composition, with $N$ generators; see Section 8.3 of Bj\"orner and Brenti \cite{BB} for a detailed look at affine permutations from this perspective.

For any given size $N>1$, there are infinitely many affine permutations of size $N$; indeed, for some patterns such as $\tau=321$, there are infinitely many affine permutations of size $N$ that avoid $\tau$.  One can view the following definition, which 
we introduced in our companion paper \cite{MT1}, 
as a reasonable attempt to make these sets finite, 
but
%  and in Section \ref{sec-motiv}, 
there are more compelling reasons for considering this definition,
as we describe below.

\begin{defi} 
   \label{def.bounded}
A \emph{bounded affine permutation} of size $N$ is an affine permutation $\sigma$ of size $N$ such that $|\sigma(i) - i| < N$ for all $i$. \end{defi}

Figure \ref{fig:boundedaffine} illustrates an example of a bounded affine permutation.

\begin{figure}
\[
\begin{tikzpicture}[scale=0.25]
\draw (1,5) [fill=black] circle (.3);
\draw (2,2) [fill=black] circle (.3);
\draw (3,4) [fill=black] circle (.3);
\draw (4,9) [fill=black] circle (.3);
\draw (5,0) [fill=black] circle (.3);
\draw (6,1) [fill=black] circle (.3);
\draw (7,11) [fill=black] circle (.3);
\draw (8,8) [fill=black] circle (.3);
\draw (9,10) [fill=black] circle (.3);
\draw (10,15) [fill=black] circle (.3);
\draw (11,6) [fill=black] circle (.3);
\draw (12,7) [fill=black] circle (.3);
\draw (13,17) [fill=black] circle (.3);
\draw (14,14) [fill=black] circle (.3);
\draw (15,16) [fill=black] circle (.3);
\draw (16,21) [fill=black] circle (.3);
\draw (17,12) [fill=black] circle (.3);
\draw (18,13) [fill=black] circle (.3);
\draw (19,23) [fill=black] circle (.3);
\draw (20,20) [fill=black] circle (.3);
\draw (21,22) [fill=black] circle (.3);
\draw (22,27) [fill=black] circle (.3);
\draw (23,18) [fill=black] circle (.3);
\draw (24,19) [fill=black] circle (.3);
\draw [thick] (-2,8) -- (27,8);
\draw [thick] (8,-2) -- (8,27);
\node at (-1,-0.5) {$\iddots$};
\node at (26,26) {$\iddots$};
\draw [dashed,thick] (-2,4) -- (21,27);
\draw [dashed,thick] (4,-2) -- (27,21);
\end{tikzpicture} \]
\caption{A bounded affine permutation of size $6$, whose values on $1,\ldots,6$ are $2, 7, -2, -1, 9, 6$. For the affine permutation to be bounded, its entries must all lie strictly between the dashed lines.}
\label{fig:boundedaffine}
\end{figure}
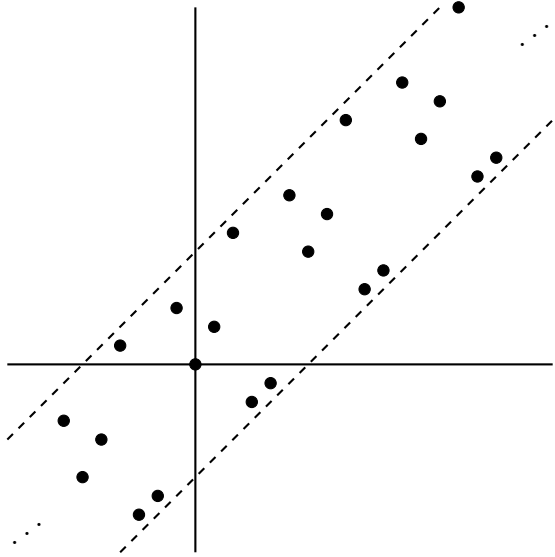

\begin{rem}
Affine permutations with a different boundedness condition were introduced by Knutson, Lam, and Speyer \cite{KLS}, who used them to study the totally non-negative Grassmannian and positroids. The bounded affine permutations in our paper are not the same as those.
\end{rem}

Let $S_N$ denote the set of permutations of size $N$, and let $\E_N$ denote the set of bounded affine permutations of size $N$.
We also define
\[        S \;:=\;  \bigcup_{N\ge0}S_N
     \hspace{5mm}\hbox{and}\hspace{5mm}
     \E   \;:=\;  \bigcup_{N\ge1}\E_N\,.
\]
In our companion paper \cite{MT1}, we find exact and asymptotic formulas for $|\E_N|$, the total number of bounded affine permutations of size $N$. We show that
\begin{equation}
    \label{eq.Enexact}
    |\E_N| \;=\; \sum_{m=0}^N \binom{N}{m} \sum_{k=0}^m \binom{m}{N-k} (-1)^{N-m} a(m,k)
\end{equation}
where $a(m,k)$ are the Eulerian numbers (the number of permutations of size $m$ with $k$ excedances), and that
\begin{equation}
    \label{eq.Enasym}
  |\E_N| \;\sim \; 
  \sqrt{\frac{3}{2\pi e N}}\, 2^N \,N! \hspace{5mm}
    \hbox{ as }N\rightarrow\infty.
\end{equation}.

If we view a permutation $\pi\in S_N$ as a
bijection on $[N]$, then we can extend it periodically by 
Equation (i) of Definition \ref{def.affine}
to a bijection $\oplus\pi$ on $\mathbb{Z}$; that is,
\[ \oplus \pi(i+kN) = \pi(i) + kN \quad \text{for $i \in [N]$ and $k \in \Z$}. \]
Observe that $\oplus \pi \in \E_N$ (see 
Figure \ref{fig.affinepi}). We call $\oplus \pi$ the \emph{infinite sum} of $\pi$. The map $\pi \mapsto \oplus \pi$ is an injection from $S_N$ into $\E_N$.

\setlength{\unitlength}{1.2mm}
\begin{figure}
% \newsavebox{\ddott}
%\savebox{\ddott}{\circle*{0.8}}
% \savebox{\dott}(0,0){\circle*{0.8}}
%\newsavebox{\ddash}
%\savebox{\ddash}{\line(1,1){1}}
  \begin{center}
%\begin{picture}(140,150)
%
%\put(40,90){
\begin{picture}(70,70)(-25,-25)
\put(-23,0){\vector(1,0){68}}
\put(0,-23){\vector(0,1){68}}
\put(2,2){\line(1,0){20}}
\put(2,2){\line(0,1){20}}
\put(2,22){\line(1,0){20}}
\put(22,2){\line(0,1){20}}
\put(2,0){\line(0,1){1}}
\put(1,-4){$1$}
\put(22,0){\line(0,1){1}}
\put(20,-4){$N$}
\put(0,2){\line(1,0){1}}
\put(-3,1){$1$}
\put(0,22){\line(1,0){1}}
\put(-4,21){$N$}
\put(13,8){\huge{${\pi}$}}
\put(2,8){\circle*{2}}
\put(4,2){\circle*{2}}
\put(6,16){\circle*{2}}
\put(8,22){\circle*{2}}
\put(22,6){\circle*{2}}
%
%\multiput(-22,-2)(4,4){17}{$\line(1,1){1}$}
%\put(-22,-2){\textcolor{blue}{\line(1,1){47}}}
%\put(-2,-22){\textcolor{blue}{\line(1,1){47}}}
\put(-22,-2){\line(1,1){47}}
\put(-2,-22){\line(1,1){47}}
\put(24,0){\line(0,1){1}}
\put(24,-4){$N{+}1$}
\put(42,0){\line(0,1){1}}
\put(40,-4){$2N$}
\put(24,24){\line(1,0){20}}
\put(24,24){\line(0,1){20}}
\put(24,44){\line(1,0){20}}
\put(44,24){\line(0,1){20}}
\put(26,30){\large{copy of ${\pi}$}}
\put(24,30){\circle*{2}}
\put(26,24){\circle*{2}}
\put(28,38){\circle*{2}}
\put(30,44){\circle*{2}}
\put(44,28){\circle*{2}}
\put(-20,-20){\line(1,0){20}}
\put(-20,-20){\line(0,1){20}}
\put(-20,0){\line(1,0){20}}
\put(0,-20){\line(0,1){20}}
\put(-18,-14){\large{copy of ${\pi}$}}
\put(-20,-14){\circle*{2}}
\put(-18,-20){\circle*{2}}
\put(-16,-6){\circle*{2}}
\put(-14,0){\circle*{2}}
\put(0,-16){\circle*{2}}
%
%\put(12,-25){$\cdots \oplus\sigma\oplus \sigma\oplus\cdots  \;\; \in {\bf BA}_N$}
%
\end{picture}
\caption{Schematic plot of a permutation $\pi\in S_N$ and 
its periodic extension $\oplus \pi \in \E_N$.  For an affine
permutation of size $N$ to be bounded, all points of the 
plot must lie on or between the two diagonal lines.
\label{fig.affinepi}}  
  \end{center}
\end{figure}
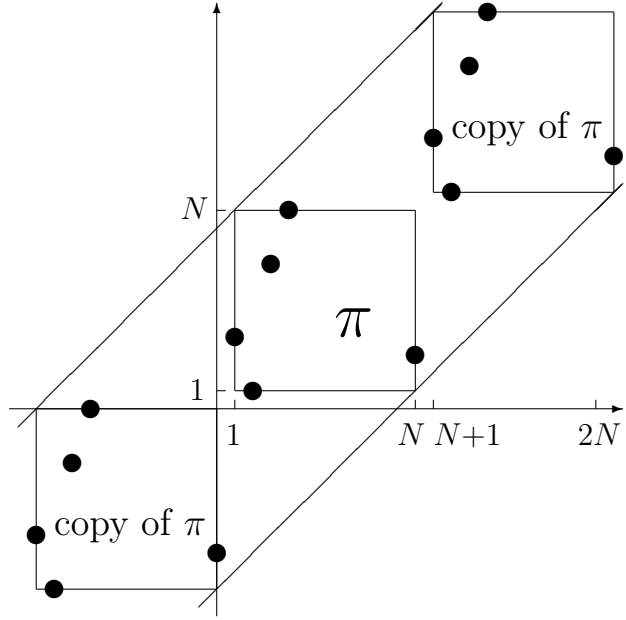

This paper concerns the set of bounded affine permutations that avoid an (ordinary) permutation $\tau$; this set is denoted $\AvBA(\tau)$, and we define pattern avoidance and related notions in Section \ref{sec-definitions}.

Let $\tau \in S_k$. It is routine to check that, if $\tau_1>\tau_k$ (or more generally if $\tau$ is sum-indecomposable), then $\sigma\oplus \pi$ avoids $\tau$ whenever $\sigma$ and $\pi$ both avoid $\tau$.
Thus the injection $\pi\mapsto \oplus\pi$ mentioned 
above is also an injection from $S_N(\tau)$ into $\E_N(\tau)$.
This proves that $|S_N(\tau)|\,\leq\,|\E_N(\tau)|$ whenever $\tau$
is sum-indecomposable.
It is harder to find a good general upper bound for $|\E_N(\tau)|$.
We posed the following conjecture in \cite{MT1}.

\begin{conj}
    \label{conj1}
The proper growth rate 
$\gr{\E(\tau)}:=\lim_{N\rightarrow\infty}|\E_N(\tau)|^{1/N}$ 
exists and equals the Stanley--Wilf limit 
$L(\tau):=\lim_{N\rightarrow\infty}|S_N(\tau)|^{1/N}$
for every sum-indecomposable pattern $\tau$.  
\end{conj}

We remark that the indecomposability condition in the conjecture is 
important; e.g.\ the only affine permutation that avoids $2143$ is the identity permutation.
In our companion paper \cite{MT1}, we prove that the conjecture holds for some specific choices of $\tau$ --- and the results of this paper show that it holds when $\tau$ is a decreasing pattern --- 
but in general we cannot
even prove that the proper growth rate $\gr{\E(\tau)}$ exists.
At least
it is easy to show that the upper growth rate $\upgr{\E(\tau)}$ is
always finite: in fact, in the companion paper we show that $\upgr{\E(\tau)} \leq \,3L(\tau)$,
where $L(\tau)$ is the Stanley--Wilf limit.

% \begin{prop}
%    \label{prop.3L}
% Let $\tau$ be a pattern.  Then
% $\limsup_{n\rightarrow\infty}|\E_n(\tau)|^{1/n} \,\leq \,3L(\tau),$ 
% where $L(\tau)$ is the Stanley--Wilf limit.
% \end{prop}
% \begin{proof}
% Observe that for any $\sigma\in\E_n$, the set of
% integers $M(\sigma):=\{\sigma(i):i\in[n]\}$ must all be
% distinct modulo $n$.  The boundedness condition tells us 
% that $M(\sigma)\subseteq (-n,2n)$, so there are at most $3^n$
% possible sets that $M(\sigma)$ could be as $\sigma$ 
% varies over $\E_n$.

% Given $\sigma\in\E_n(\tau)$, let $\sigma^{\dagger}$ be the
% ordinary permutation in $S_n(\tau)$ 
% that is order-isomorphic to the string $(\sigma(1),\sigma(2),\ldots,\sigma(n))$.
% Then $M(\sigma)$ and $\sigma^{\dagger}$ determine $\sigma$:
% indeed, the value of $\sigma(i)$ must be the 
% $[\sigma^{\dagger}(i)]^\text{th}$ smallest element of $M(\sigma)$.  It follows that $|\E_n(\tau)|\leq 3^n|S_n(\tau)|$.
% \end{proof}

In this paper, we focus on the 
avoidance of monotone decreasing patterns $m(m-1)\cdots 321$
in bounded affine permutations.  More specifically, our
first main result (Theorem \ref{thm.E4321}) is that for
every $m\geq 3$ we have the asymptotic behaviour
% there is an explicit constant $K_m$ such that 
\begin{equation} 
   \label{eq.Enm1}
   \left|\E_N(m(m-1)\cdots 321)\right|  \;\sim\;
    A_m N^{(m-2)/2}(m-1)^{2N}   \hspace{5mm}\hbox{as }
    N\rightarrow\infty
\end{equation}
where the constant $A_m$ is given by 
\begin{equation}
    \label{eq.Amdef}
    A_m \;=\;  \frac{\sum_{j=0}^{\lfloor (m-1)/2\rfloor} (-1)^j
    \binom{m-1}{j} (m-2j-1)^{m-2} 
    }{(4\pi)^{(m-2)/2}\,(m-1)^{(m-1)/2} \left[(m-2)!\right]^2 }\,.
\end{equation}
%We also obtain scaling limits as $n\to\infty$ of random 
%elements of $\E_n({\bf m(m-1)\cdots 321})$, in the framework
%of convergence of random measures that has been used in the context of permutons \cite{Glebov}.
%
(See Remark \ref{rem.Zk} and the subsequent discussion for
comments on the form this result.)
The key to proving (\ref{eq.Enm1}) is a counting argument
based on the decomposition of any member of
$\E_N(m(m-1)\cdots 321)$ into $m-1$ increasing (periodic) 
subsequences (Proposition \ref{prop.unionincreasing}).
It turns out that unlike the situation for ordinary 
permutations avoiding monotone patterns, these $m-1$
subsequences are typically 
well separated in the bounded affine case,
as represented schematically in Figure \ref{fig.321} in 
the case $m=3$.
\input{Fig321}
Indeed, in the 
plot of a random member of $\E_N(m(m-1)\cdots 321)$,
it is highly likely that each of the $m-1$ subsequences 
is confined to a narrow diagonal strip, and that the
points are approximately uniformly 
distributed within that strip in a sense that we shall make
precise in Section \ref{sec.lowerbound}.   
In addition, each subsequence
is likely to have approximately $N/(m-1)$ points 
with first coordinate in $[1,N]$.   This all suggests that 
as we let $N$ tend to infinity, the plot (scaled down 
by a factor of $N$) looks more and 
more like $m-1$ solid lines of slope 1 
(Figure \ref{fig.321limit}).  
\input{Fig321limit}
Such a phenomenon can be conveniently described in the 
framework of weak convergence of probability measures in the plane,
exactly as in the context of permutons (\cite{Glebov,Hoppen}).
We describe our framework in Section \ref{sec.weak}.  
Our second main result says that 
for fixed $m$, the scaling limit as $N\rightarrow\infty$ 
of a random element of $\E_N(m(m-1)\cdots 321)$
(viewed as an atomic measure on the plane) 
is a uniform measure 
\footnote{proportional to one-dimensional Lebesgue measure}
on $m-1$ parallel lines of slope 1 with $y$-intercepts that
are randomly chosen from $[-1,1]$, independently except
for the condition that their sum is 0.   Theorem
\ref{thm.wassmain} is a precise statement of this result.

Our motivation for initiating the study of bounded affine 
permutations is described in \cite{MT1}.  
Briefly, it is our attempt to impose an analogue of
``periodic boundary conditions'' on the plots of random
$\tau$-avoiding (ordinary) permutations for patterns such as
$\tau=4321$ or $\tau=4231$, inspired by 
Clisby's work on self-avoiding walks \cite{Clis}.  We anticipate
that (a part of) the plot of a random member of $\E_N(\tau)$ 
in some sense looks
like the \textit{middle} of the plot a random member of $S_M(\tau)$
(for some suitable $M$), far from the ``boundary effects'' that
come into play near the corners of the square $[1,M]^2$ 
and constrain the plot of an ordinary permutation
(see Figures \ref{fig.4321}--\ref{fig.4231aff}).
\begin{figure}
  \centering
    \includegraphics[clip, trim=2.5cm 8cm 1.5cm 8cm, scale=0.95]{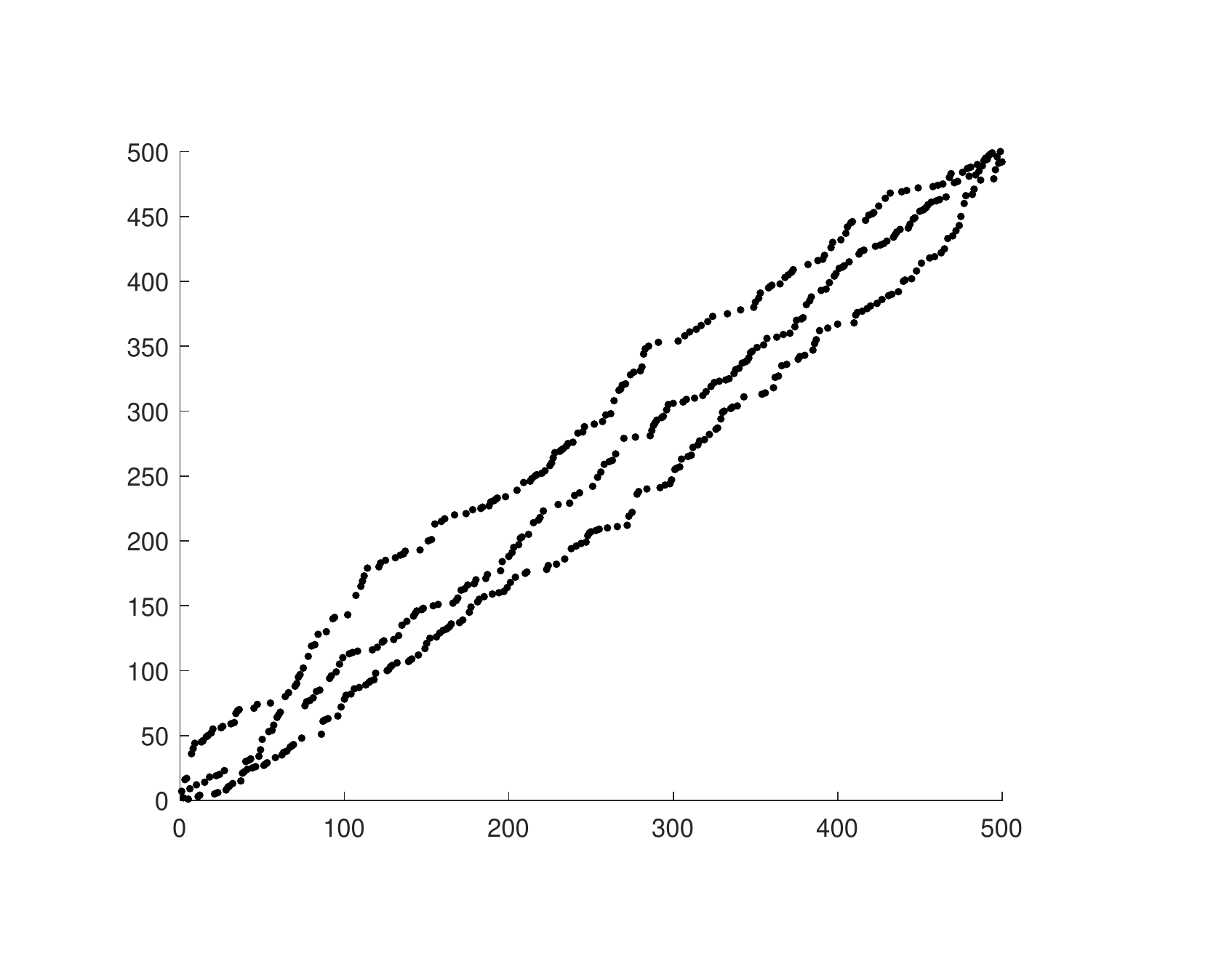}
    \caption{A random 4321-avoiding permutation
    of size 500.  This was generated by 
    G\"{o}khan Y{\i}ld{\i}r{\i}m
    using a Markov chain Monte Carlo algorithm.
    Our motivation for the present work
    was the belief that the part of this plot in 
    the strip $200<x<300$, say, should look 
    like part of the plot of a 4321-avoiding
    bounded affine permutation.}
    \label{fig.4321}
\end{figure}
\begin{figure}
    \centering
    \includegraphics[scale=0.45]{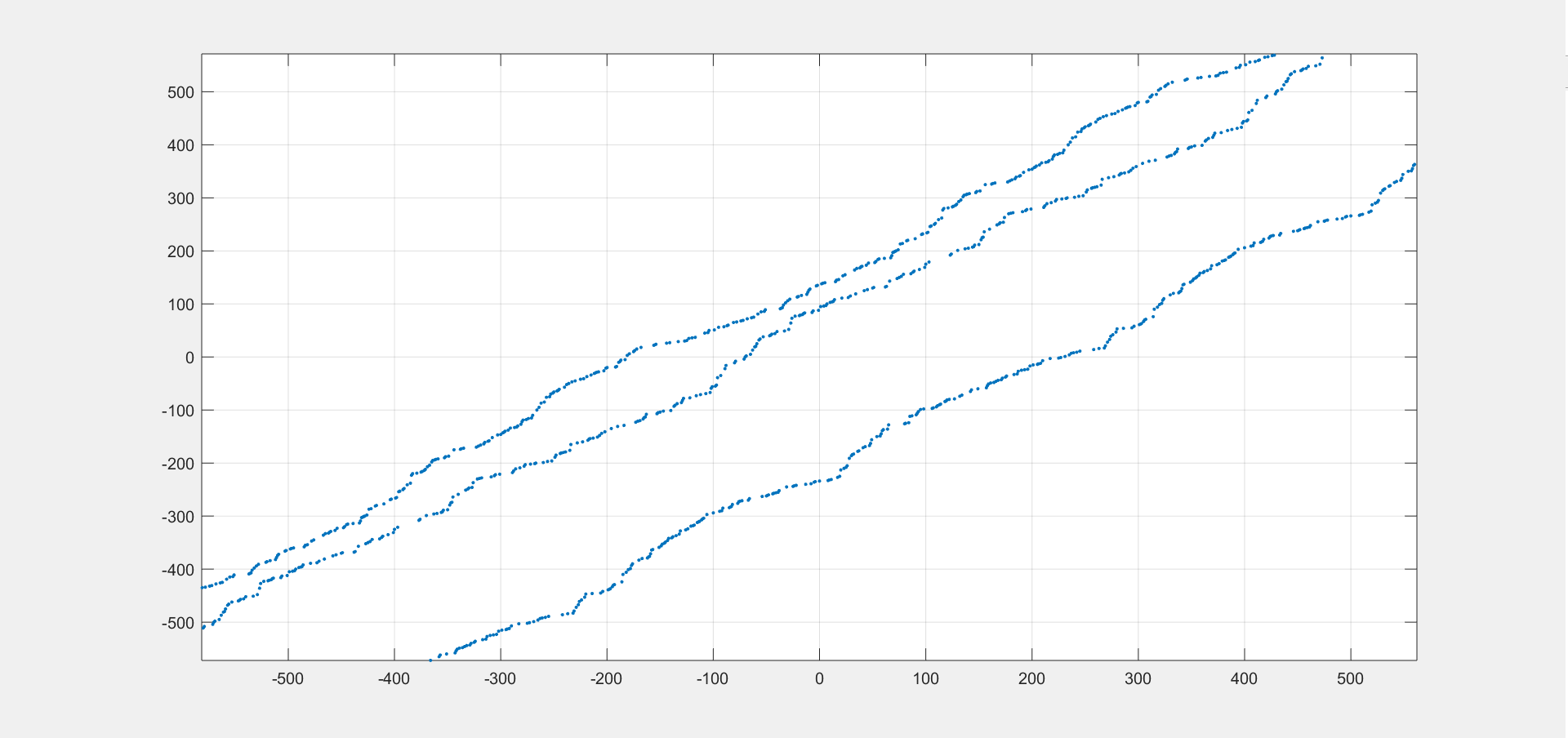}
    \caption{A random 4321-avoiding bounded affine 
    permutation of size 500.  This was generated by
    Quynh Vu using a Markov chain Monte Carlo algorithm, under the supervision of the 
    first author.}
    \label{fig.4321aff}
\end{figure}
\begin{figure}
  \centering
    \includegraphics[clip, trim=2.5cm 8cm 1.5cm 8cm, scale=0.95]{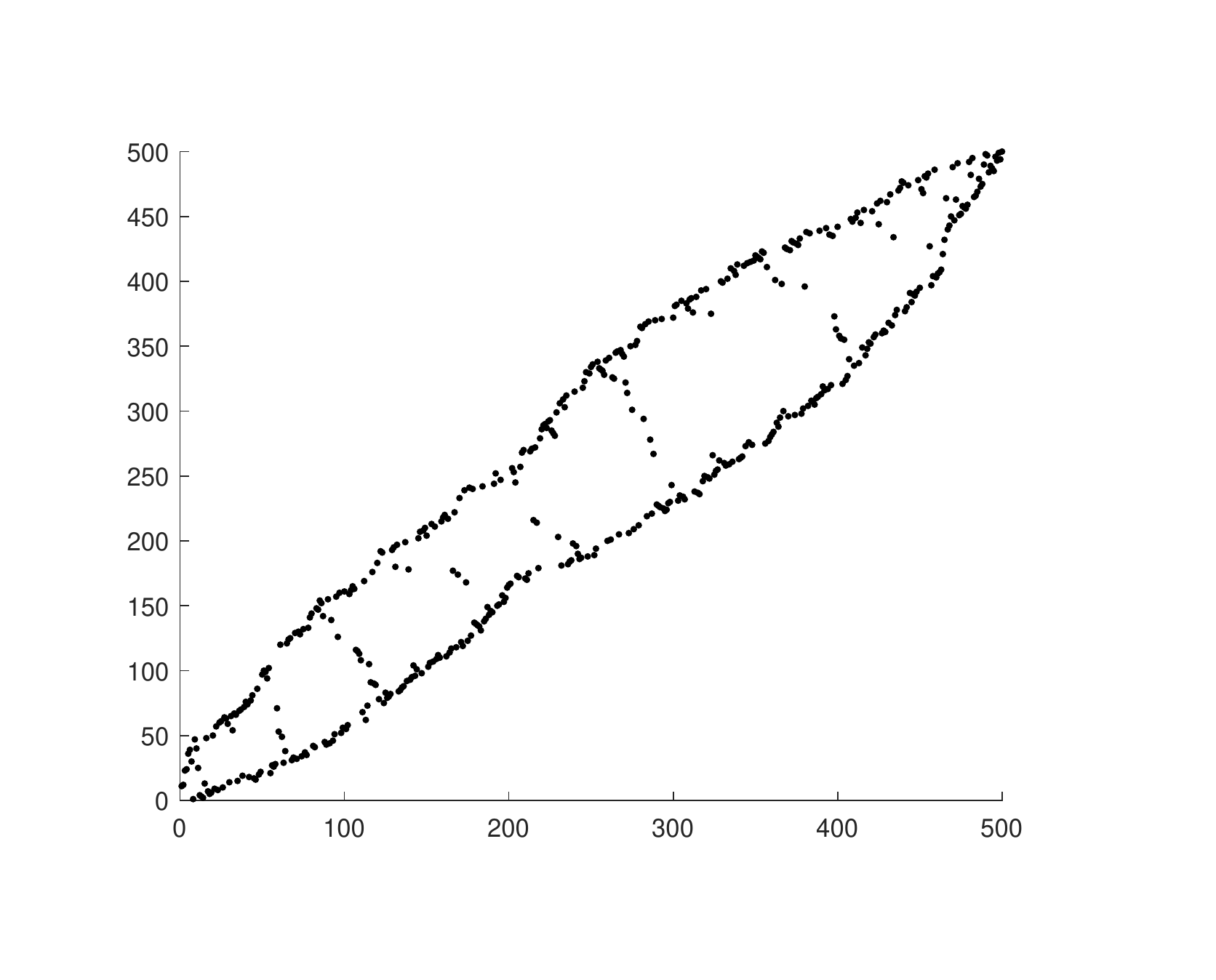}
    \caption{A random 4231-avoiding permutation
    of size 500.  This was generated by 
    G\"{o}khan Y{\i}ld{\i}r{\i}m
    using a Markov chain Monte Carlo algorithm.
    Our motivation for the present work
    was the belief that the part of this plot in 
    the strip $200<x<300$, say, should look 
    like part of the plot of a 4231-avoiding
    bounded affine permutation.}
    \label{fig.4231}
\end{figure}
\begin{figure}
    \centering
    \includegraphics[clip, trim=1cm 0cm 1cm 0cm, scale=0.50]{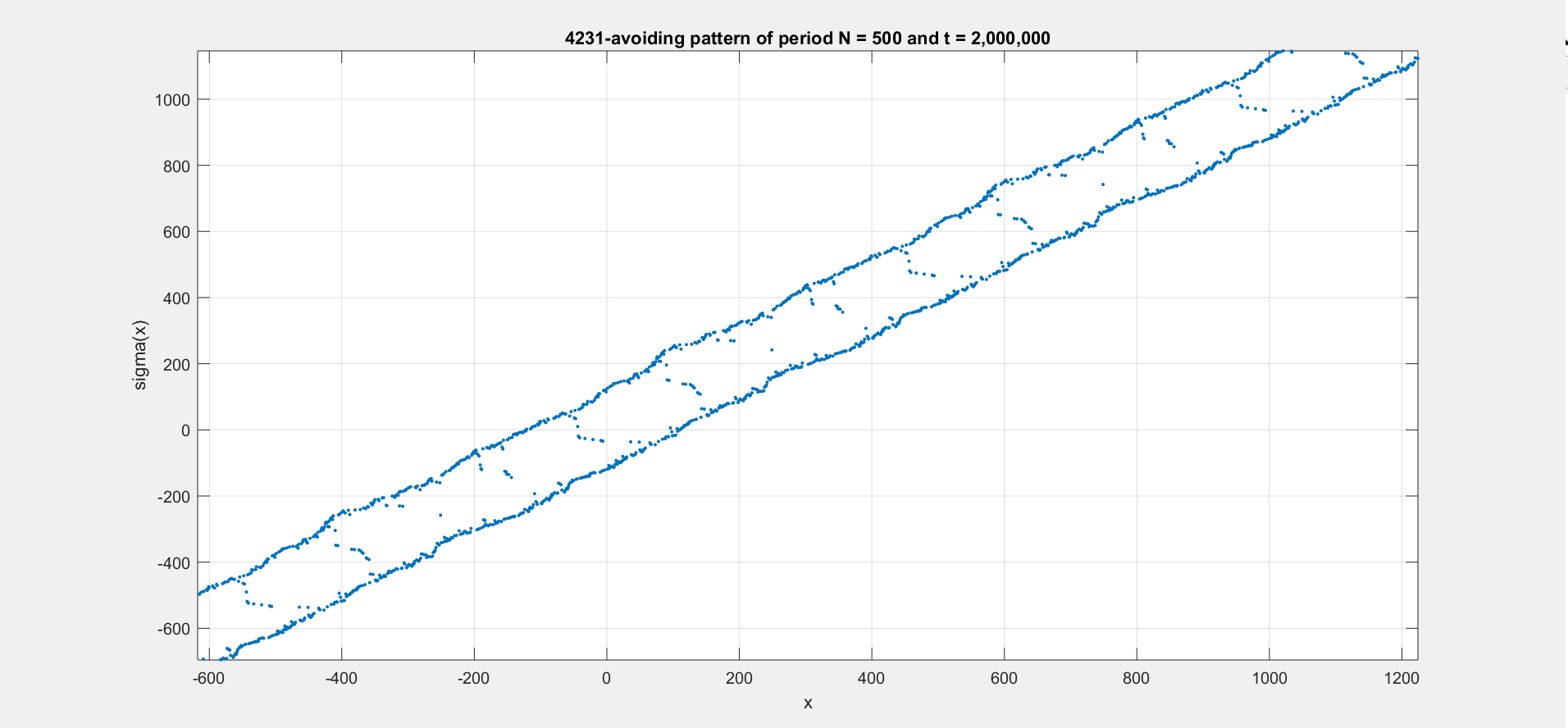}
    \caption{A random 4321-avoiding bounded affine 
    permutation of size 500.  This was generated by
    Quynh Vu using a Markov chain Monte Carlo algorithm, under the supervision of the 
    first author.}
    \label{fig.4231aff}
\end{figure}

\subsection{Definitions and notation}
\label{sec-definitions}

For sequences $\{a_n\}$ and $\{b_n\}$, we write 
$a_n\sim b_n$ to mean $\lim_{n\to\infty} a_n/b_n \,=\, 1$.
For $n\in \mathbb{N}$, we write $[n] = \{1, \ldots, n\}$.
We denote the Euclidean norm by $||\cdot||$.

Affine permutations and bounded affine permutations were defined above. We let $S_N$ denote the set of permutations of size $N$, and we let $\E_N$ denote the set of bounded affine permutations of size $N$. Furthermore, we set $S = \bigcup_{N\ge0} S_N$ and $\E = \bigcup_{n\ge1} \E_N$. 

We represent an ordinary permutation $\sigma\in S_N$ either as a
function $\sigma:[N]\rightarrow [N]$ or as a finite sequence
$\sigma_1\sigma_2\ldots\sigma_N$ where $\sigma_i=\sigma(i)$.
For affine permutations, we only use the function notation.

We begin by introducing concepts that are standard in permutation patterns research. The \emph{diagram} or
\emph{plot} of a permutation $\pi \in S_N$ is the set of points $\{(i, \pi(i))\,:\,i \in [N]\}$. Given permutations $\pi$ and $\tau$, we say that \emph{$\pi$ contains $\tau$ as a pattern}, or simply that \emph{$\pi$ contains $\tau$}, if the diagram of $\tau$ can be obtained by deleting zero or more points from the diagram of $\pi$ (and shrinking corresponding segments of the axes), i.e.\ if $\pi$ has a subsequence whose entries have the same relative order as the entries of $\sigma$. We may also say that two sequences with the same relative order are \emph{order-isomorphic}. We say \emph{$\pi$ avoids $\tau$} if $\pi$ does not contain $\tau$. For instance, for $\pi = 493125876$, the subsequence $9356$ is an occurrence of $\tau = 4123$, but on the other hand $\pi$ avoids $3142$. See Figure \ref{fig:contain}.

\begin{figure}
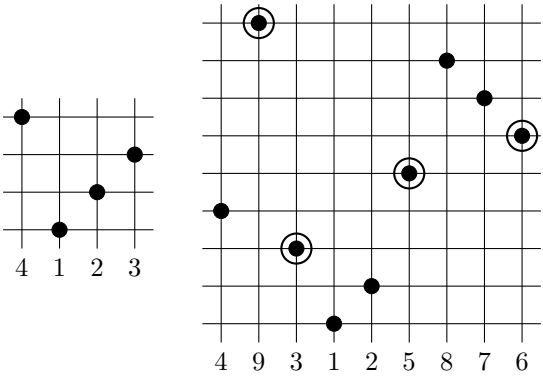

\[ \drawpermutation{4,1,2,3}{4} \hspace{0.25in}
        \drawpattern{4,9,3,1,2,5,8,7,6}{9}{9,3,5,6} \]
    \caption{The permutation $4123$ is contained in the permutation $493125876$.}
    \label{fig:contain}
\end{figure}

If $\tau$ is a permutation, then $\Av{(\tau)}$ denotes the set of all permutations that avoid $\tau$, and $\Av_N{(\tau)}$ is the set of such permutations of size $N$, i.e.\ $\Av_N({\tau}) = \Av{(\tau)} \cap S_N$. The \emph{upper growth rate} of $\Av{(\tau)}$ is defined as $\upgr{\Av{(\tau)}}:= \limsup_{N\to\infty} |\Av_N{(\tau)}|^{1/N}$, and the \emph{lower growth rate} is defined as 
$\logr{\Av{(\tau)}}:=\liminf_{N\to\infty} |\Av_N{(\tau)}|^{1/N}$. If the upper and lower growth rates of $\Av{(\tau)}$ are equal, i.e.\ if $\lim_{N\to\infty} |\Av_N{(\tau)}|^{1/N}$ exists (or is $\infty$), then this number is called the \emph{proper growth rate} of $\Av{(\tau)}$, written $\gr{\Av{(\tau)}}$. 
By the Marcus--Tardos Theorem \cite{MarTar} (formerly the Stanley--Wilf Conjecture), $\Av{(\tau)}$ has a finite upper growth rate for every $\tau$. It is also known that $\Av{(\tau)}$ has a proper growth rate for every $\tau$ (proved by Arratia \cite{Arr}); this growth rate is often called the Stanley--Wilf limit and denoted $L(\tau)$.

%Given $\sigma \in S_a$ and $\tau \in S_b$, the \emph{sum} $\sigma \oplus \tau$ is the permutation in $S_{a+b}$ obtained by juxtaposing the diagrams of $\sigma$ and $\tau$ diagonally: that is, $(\sigma \oplus \tau)(i) = \sigma(i)$ if $1 \le i \le a$, and $(\sigma \oplus \tau)(i) = a+\tau(i-a)$ if $a+1 \le i \le a+b$. (This explains the notation $\oplus \sigma$ defined above, which is the doubly infinite sum of $\sigma$ with itself.) A class $\CC$ is \emph{sum closed} if $\sigma, \tau \in \CC$ implies $\sigma \oplus \tau \in \CC$. A permutation is \emph{sum-indecomposable}, or \emph{indecomposable}, if it is not the sum of two permutations of non-zero size. The set of indecomposable permutations in a class $\CC$ is denoted $\ind{\CC}$.

We now introduce the analogous concepts for affine permutations. The \emph{diagram} or \emph{plot} of an affine permutation $\omega \in \E_N$ is the set of points $\{(i, \omega(i)) \,:\,i \in \Z\}$. Given an affine permutation $\omega$ and an ordinary permutation $\tau$, we say that \emph{$\omega$ contains $\tau$ as a pattern}, or simply that \emph{$\omega$ contains $\tau$}, if the diagram of $\tau$ can be obtained by deleting some points from the diagram of $\omega$, i.e.\ if $\omega$ has a subsequence whose entries have the same relative order as the entries of $\tau$. We say \emph{$\omega$ avoids $\tau$} if $\omega$ does not contain $\tau$, and we let $\AvBA{(\tau)}$ denote the set of all bounded affine permutations that avoid $\tau$. The idea of an affine permutation containing or avoiding a given ordinary permutation was first used by Crites \cite{Crites}.

We can define $\upgr{\AvBA{(\tau)}}$, $\logr{\AvBA{(\tau)}}$, and $\gr{\AvBA{(\tau)}}$ for bounded affine permutations in the same way as for ordinary permutations, though we do not know whether $\gr{\E(\tau)}$ exists for every ordinary permutation $\tau$, as it does in the setting of ordinary permutation classes. As we noted above, $\upgr{\E(\tau)}$ is always finite.

Note that, if $d \mid N$, then every element of $\E_d$ is also an element of $\E_N$. If $\omega$ is an affine permutation of size $N$, then $N$ need not be the smallest possible size of $\omega$. Thus, for enumeration purposes, our count of affine permutations of size $N$ with a given property includes the affine permutations of size $d$ with that property for $d \mid N$.

% START OF NEW MATERIAL

\section{Avoiding a decreasing pattern:  Enumeration}

It is well known that a permutation avoids the decreasing pattern $m(m{-}1)\cdots 21$ if and only if it can be 
partitioned into $m{-}1$ increasing subsequences.  It is also true that an affine permutation
avoids $m(m{-}1)\cdots 21$ if and only if it can be partitioned into $m-1$ \emph{periodic} increasing subsequences. Since the number of increasing subsequences is more fundamental to our
development than is the length of the pattern, we shall write $k$ for $m{-}1$ in our work, 
and state our results with $m$ replaced by $k{+}1$. We denote the decreasing permutation of size $k+1$ by $\deck$. That is:

\begin{prop} \label{prop.unionincreasing}
Let $\omega$ be an affine permutation of size $N$, and assume that $N \ge k$. Then $\omega$ avoids $\deck$ if and only $[N]$ can be partitioned into $k$ non-empty sets, $[N] = G_1 \cup \cdots \cup G_k$, such that for each $i$ if $G_i = \{g_{i,1} < g_{i,2} < \cdots < g_{i,n_i}\}$ (where $n_i = |G_i|$) then $\omega(g_{i,1}) < \omega(g_{i,2})  < \cdots < \omega(g_{i,n_i}) < \omega(g_{i,1}+N)$. 
\end{prop}

\begin{proof}
Just as in the case of ordinary permutations, it is clear that, if there exists a partition $[N] = G_1 \cup \cdots \cup G_k$ satisfying the conditions from the proposition statement, then $\omega$ avoids $\deck$: indeed, the positions in an occurrence of $\deck$ would have to be in $k+1$ different increasing subsequences.

The converse is proved by the same method as in the classical version for ordinary permutations (see \cite[Thm.\ 4.10]{BonaCP}). Assume $\omega$ avoids $\deck$. For each $a \in \Z$, define the \emph{rank} of $a$ (in $\omega$) to be the maximum number $r$ such that $a$ is the start of a sequence of $r$ integers $a = a_1 < a_2 < \cdots < a_r$ such that $\omega(a_1) > \omega(a_2) > \cdots > \omega(a_r)$. That is, the rank of $a$ is the maximum length of a decreasing subsequence of $\omega$ that begins with position $a$. Since $\omega$ avoids $\deck$, every integer has rank $r$ satisfying $1 \le r \le k$. By the definition of affine permutation, $a$ and $a+N$ have the same rank for all $a \in \Z$.

For each $i \in [k]$ define $\widetilde{G}_i$ to be the set of integers of rank $i$. Then $\Z = \widetilde{G}_1 \cup \cdots \cup \widetilde{G}_k$ is a partition of $\Z$, possibly with some blocks empty, with the property that $a \in \widetilde{G}_i$ if and only if $a+N \in \widetilde{G}_i$. For each $i$ such that $\widetilde{G}_i$ is non-empty, $\{\omega(a)\}_{a \in \widetilde{G}_i}$ is a doubly infinite increasing subsequence of $\omega$ (meaning if $a,a' \in \widetilde{G}_i$ and $a < a'$ then $\omega(a) < \omega(a')$).

Finally, if we define $G_i = \widetilde{G}_i \cap [N]$, then $[N] = G_1 \cup \cdots \cup G_k$ is a partition of $[N]$ satisfying the conditions given in the proposition statement, except that some $G_i$ may be empty. This last detail can be corrected by removing the empty blocks and subdividing the non-empty blocks until there are exactly $k$ of them (this is possible because $N \ge k$).
\end{proof}

%In this section we consider ${\cal E}_N(4321)$.  The method here should generalize directly to 
%${\E}_N(\deck)$, but I will work with $k=3$ for now until we are sure that everything is 
%correct.  The proof works also for $k=2$, but differs in some ways from the proof for 321 in the 
%previous sections.  The ideas are the same, but I think this new way is easier even for the 321 case.

Here is the first of the two main theorems of this paper.  
Everything is trivial for $k=1$, 
so in the rest of the paper we shall always assume $k\geq 2$.

\begin{thm}
   \label{thm.E4321}
Fix $k \ge 1$. As $N \to \infty$,
\begin{equation}
    \label{eq.E4321}
      |  {\E}_N(\deck)|   \;\sim\;   k^{2N} \left( \frac{N}{4\pi}\right)^{(k-1)/2} \frac{Z^*_k}{k^{k/2}(k-1)!}  
\end{equation}
where 
\begin{equation}
  \label{eq.Zstardef}
       Z^*_k   \;=\;  \frac{1}{(k-1)!}\sum_{j=0}^{\lfloor k/2 \rfloor}  (-1)^j  \binom{k}{j}\left(  k-2j\right)^{k-1}.
\end{equation}
\end{thm}

\begin{rem}  \label{rem.Zk}
We note that $Z^*_k/2^{k-1}$ is the value of the probability density function
of the sum of $k$ independent uniform random variables on 
$[0,1]$ evaluated at its midpoint, $k/2$; see ``Irwin--Hall distribution'' in \cite[Sec.\ 26.9, Eq.\ (26.48)]{JKB}.
We easily compute $Z_1^*=1$, $Z_2^*=2$, $Z_3^*=3$, $Z_4^*=16/3$, 
and $Z_5^*=115/12$.
\end{rem}

For instance, for $k=2$ this becomes
\[ |\E_N(321)| \sim \sqrt{\frac{N}{4\pi}} \cdot 4^N, \]
and for $k=3$ we obtain
\[ |\E_N(4321)| \sim \frac{N}{8\pi\sqrt{3}} \cdot 9^N. \]
For every $k$, the proper growth rate of $\E(\deck)$ is $k^2$, the same as for ordinary permutations avoiding $\deck$.
More precisely, Regev \cite{Regev} showed that for the latter,
\[    |S_N(\deck)|  \;\sim\; I_k\, k^{2N} N^{-(k^2-1)/2}
\]
where 
\[  I_k \;=\;  \frac{1}{(2\pi)^{k}(k+1)^{(k+1)^2}(k+1)!}
    \int_{\R^{k+1}}
   \prod_{i<j}(x_i-x_j)^2 e^{-(k+1)(x_1^2+\cdots+x_{k+1}^2)}
   \,dx_1\cdots dx_{k+1} \,.
\]
We remark that $|\E_N(\deck)|\,/\, |S_N(\deck)|$ is asymptotically
proportional to $N^{(k^2+k-2)/2}$ as $N\to\infty$.

\subsection{The setup}

Here is the setup that we will use to prove Theorem \ref{thm.E4321}, relying on the characterization from Proposition \ref{prop.unionincreasing} that a permutation avoids $\deck$ if and only if it can be expressed as
the union of $k$ increasing subsequences.

Given positive integers $n_1,\ldots,n_k$ whose sum is $N$, let $\{G_1,\ldots,G_k\}$ 
and $\{H_1,\ldots,H_k\}$ be two partitions of $\{1,\ldots,N\}$ such that $|G_i|=|H_i|=n_i>0$ for each $i\in [k]$.
For each $i$, write the elements of the sets $G_i$ and $H_i$ as
\begin{align}
   \nonumber
   G_i  \;=\; \{g_{i,1},g_{i,2},\ldots,g_{i,n_i}\}  & \hspace{5mm}\hbox{where}\hspace{5mm}
   g_{i,1}\,<\,g_{i,2} \,< \,\ldots \,<\, g_{i,n_i}   \,,
   \\
   \label{eq.horder}
  H_i  \;=\; \{h_{i,1},h_{i,2},\ldots,h_{i,n_i}\}  & \hspace{5mm}\hbox{where}\hspace{5mm}
   h_{i,1}\,<\,h_{i,2} \,< \,\ldots \,<\, h_{i,n_i}   \, .
\end{align}
Finally, let $\Delta_1,\ldots,\Delta_k$ be integers such that 
\begin{equation}
    \label{eq.Deltasum}
   \Delta_i \in[-n_i,n_i] \quad (i\in[k])
   \hspace{5mm}\mbox{and}\hspace{5mm}
      \sum_{i=1}^k\Delta_i \;=\; 0 \,. 
\end{equation}
To shorten the notation,  we shall write $\collec{n}$ to represent the ordered $k$-tuple $(n_1, \ldots, n_k)$, and similarly for 
$\collec{G}$, $\collec{H}$, and $\collec{\Delta}$. The procedure described in the next several paragraphs will define a function $\Psi$ whose domain $\mathcal{D}_0(N)$ is the set of all $(4k)$-tuples $(\collec{n}, \collec{G}, \collec{H}, \collec{\Delta})$ that satisfy the conditions just described, and whose codomain contains ${\E}_N(\deck)$. 
The correspondence $\Psi$ is the key to our main theorem, as
we shall outline soon. 
% \justin{How are these sentences I added?}  \neal{Good.}

At this point it is useful to pause and observe that we can use $\collec{G}$ and $\collec{H}$ to define an ordinary permutation 
$\sigma$ in ${\cal S}_N(\deck)$ by specifying 
\begin{equation}
    \label{eq.ghord}
     \sigma(g_{i,j}) \;=\;h_{i,j}    \hspace{5mm}\mbox{for $j=1,\ldots,n_i$ and $i\in [k]$}.    
\end{equation}
See Figure \ref{fig.ABAB}.
\input{FigABAB} 
Every permutation in ${\cal S}_N(\deck)$ can be created this way, but not uniquely.  One obvious
source of non-uniqueness
is that we can permute the subscripts of $n_i$, $G_i$, and $H_i$ in $k!$ ways and get the 
same $\sigma$.  This leads to the bound
\begin{equation}
   \label{eq.S4321bd}
    |{\cal S}_N(\deck)|   \;\leq\;   \frac{1}{k!}  \sum_{\substack{n_1,\ldots,n_k\geq 1\\n_1+\cdots+n_k=N}}
      \left( \begin{array}{c}  N  \\ n_1,\,n_2,\,\ldots,\,n_k \end{array}  \right)^2.
\end{equation}
We note that the permutation of subscripts is not the only reason that the above association is not unique. For example, we can get the identity permutation by taking $H_i=G_i$ for  any choice of 
$\{G_1,\ldots,G_k\}$.  Also, notice that if $\sigma(1)=1$ and $1\in G_1$, say, then moving the 
element 1 from $G_1$ to  $G_2$ and moving 1 from $H_1$ to $H_2$ gives a different decomposition
of the same $\sigma$ into $k$ increasing parts.  
For the case $k=2$, the upper bound of Equation (\ref{eq.S4321bd})
becomes $\frac{1}{2}\sum_{k=1}^{N-1}\binom{N}{k}^2 \,=\, \frac{1}{2}\binom{2N}{N}-1$, which is an order of $N$ larger
than the correct answer, $|{\cal S}_N(321)|=\binom{2N}{N}/(2N+1)$
(e.g.\ Corollary 4.7 of \cite{BonaCP}).  In contrast, the analogous 
bound that we shall derive for $|\E_N(\deck)|$ will be asymptotically
exact.

Now we describe the procedure that defines the function $\Psi$, which will take a $(4k)$-tuple $(\collec{n}, \collec{G}, \collec{H}, \collec{\Delta})$ in its domain  $\mathcal{D}_0(N)$
and use it to construct an affine permutation.
The asymptotic upper bound of Section \ref{sec.upperbound}
comes from the fact
that each permutation in $\E_N(\deck)$ has at least $k!$
preimages in $\mathcal{D}_0(N)$ under $\Psi$ (this leads to Equation
(\ref{eq.4upbound})).  
The matching asymptotic lower bound of Section \ref{sec.lowerbound}
relies on finding a slightly 
smaller domain $\Dom$ (depending on $N$ as well as some
other parameters) that $\Psi$ maps into $\E_N(\deck)$, and on 
which $\Psi$ is exactly $k!$-to-one.  Indeed, plots such as those
suggested by Figure \ref{fig.321} are images of members of 
$\Dom$.

We first define some notation as well as
the domain $\mathcal{D}_0(N)$ of $\Psi$.

\begin{defi} 
   \label{def.domains}
(a) For natural numbers $w$ and $N$, let $\Seq(w,N)$ be the set 
of all
$w$-element subsets of $[N]$.  We shall typically identify such 
a set as an increasing subsequence of $1,2,\ldots,N$, as we do 
in Equation (\ref{eq.horder}).
\\
(b) For $\collec{n}\in \mathbb{N}^k$, let 
\[   D_{\Delta}(\collec{n})  \;=\;  \left\{  \collec{\Delta}\in 
   \mathbb{Z}^k  \,:\,  |\Delta_i|\leq n_i \quad \forall i\in [k],
    \textup{ and }\sum_{i=1}^k\Delta_i=0 
    \right\} \,.
\]
Also, let $Z(n_1,\ldots,n_k)=|D_{\Delta}(\collec{n})|$.
\\
(c) For $N\in \mathbb{N}$, define the following set of ($4k$)-tuples: \begin{eqnarray*}
   \mathcal{D}_0(N) & = &   \left\{
   (\collec{n},\collec{G},\collec{H},\collec{\Delta})\,:\, 
   \collec{n}\in \mathbb{N}^k, \;\sum_{i=1}^kn_i=N, \;
    \collec{\Delta}\in D_{\Delta}(\collec{n}), \, \right.  \\
    & &  \hspace{5mm}  \left. \collec{G} \textup{ and $\collec{H}$
      are partitions of $[N]$ such that $|G_i|=|H_i|=n_i$ }
      \forall i\in [k] \,
     \rule{0cm}{0.6cm}  \right\}
\end{eqnarray*}
\end{defi}

We now explain how to define $\Psi$ on the domain
$\mathcal{D}_0(N)$.
Let $(\collec{n},\collec{G},\collec{H},\collec{\Delta})\in 
\mathcal{D}_0(N)$.
For each $i$,  
extend the definition of $g_{i,j}$ and $h_{i,j}$ from 
Equation (\ref{eq.horder}) to all integers $j$ periodically, i.e.
\begin{equation}
    \label{eq.hperiodic}
  g_{i,j+tn_i} \;:=\;   g_{i,j}  \,+\, tN  \hspace{5mm}\hbox{and}\hspace{5mm}
          h_{i,j+tn_i} \;:=\;   h_{i,j}  \,+\, tN      \hspace{5mm} \hbox{for $j\in [n_i]$, $t\in \mathbb{Z}$}.
\end{equation}
Observe that 
%\begin{align*}
 %    h_{i,n_i+1}   & \;  > \;  N \;\geq \;   h_{i,n_i}      \hspace{6mm}\mbox{and}    \\
 %     h_{i,1-n_i}   & \;  \geq  \;  1-N \;> \;-N  \;=\; N-2N  \; \geq  \;  h_{i,-n_i}  \,.
% \end{align*}     
\begin{equation}
    \label{eq.hincr}
    g_{i,j+tn_i} \, , \,  h_{i,j+tn_i}    \;\in\; [1+tN,N+tN]     \hspace{5mm} \hbox{for $j\in [n_i]$, $t\in \mathbb{Z}$}.
\end{equation}  
In particular, for each $i$, we see that $g_{i,j}$ is a strictly increasing function of $j\in \mathbb{Z}$, and that 
for each $x \in \mathbb{Z}$ there is a unique choice of $i$ and $j$ such that $g_{i,j} = x$ 
(and similarly for $h_{i,j}$).  
We define $\Psi(\collec{n}, \collec{G}, \collec{H},\collec{\Delta})$
to be the function $\sigma$ given by
\begin{equation}
   \label{eq.sigma4def}
        \sigma(g_{i,j})   \;=\;      h_{i,j+\Delta_i}    \hspace{5mm}  \hbox{for $j\in \mathbb{Z}$, $i\in[k]$.}
\end{equation}
See Figure \ref{fig.ABABaffine}.
\input{FigABABaffine}
We remark that if $\Delta_i=0$ for each $i\in[k]$, then $\sigma$ is
just the infinite direct sum of the ordinary permutation we 
created in Equation (\ref{eq.ghord}) above 
(recall Figure \ref{fig.ABAB}).
%We define $\Psi(\collec{n}, \collec{G}, \collec{H},\collec{\Delta})$
%to be the permutation $\sigma$.

\begin{lemma}
    \label{lem.sigmaaffine}
Let $\vec{v}=(\collec{n},\collec{G},\collec{H},\collec{\Delta})
\in \mathcal{D}_0(N)$ and let $\sigma=\Psi(\vec{v})$, 
as defined by Equations (\ref{eq.hperiodic}) and (\ref{eq.sigma4def}) above.  
Then $\sigma$ is a (not necessarily bounded) affine permutation of size $N$ that avoids $\deck$.  Moreover, every member of ${\E}_N(\deck)$
can be obtained in this way, i.e.\ every member of ${\E}_N(\deck)$ is in the image of $\Psi$.
\end{lemma}

We remark that $\sigma$ in the image of $\Psi$ is not necessarily in ${\E}_N$, since there is no guarantee that the constraint $|\sigma(i)-i|<N$ holds for all $i$.
    
\begin{proof}[Proof of Lemma \ref{lem.sigmaaffine}.]  Observe first that Equation (\ref{eq.hperiodic}) actually holds for every integer $j$.
By Equation (\ref{eq.hincr}) and the subsequent comments, it is apparent that $\sigma$ is a 
well-defined bijection of $\mathbb{Z}$.  Property ($i$) of Definition \ref{def.affine} follows from Equations 
(\ref{eq.hperiodic}) and (\ref{eq.sigma4def}).
For property ($ii$), let $f_i(r) \,=\, \sum_{j=1}^{n_i}h_{i,j+r}$ for $i\in[k]$ and $r\in \mathbb{Z}$.
Then 
\[   f_i(r+1)-f_i(r)   \;=\;    h_{i,n_i+r+1}-h_{i,r+1}   \;=\; N  \hspace{5mm}\hbox{for every $r\in \mathbb{Z}$.}
\]
Moreover, 
\[     \sum_{i=1}^k f_i(0)    \;=\;  \sum_{i=1}^k \sum_{j=1}^{n_i} h_{i,j}   \;=\;  \sum_{\ell=1}^N \ell \,.    
\]
It follows that 
\begin{align}
    \nonumber
   \sum_{\ell=1}^N\sigma(\ell)   \; & =\;    \sum_{i=1}^k\sum_{j=1}^{n_i} \sigma(g_{i,j})   
       \;=\;  \sum_{i=1}^k f_i(\Delta_i) \\
       \nonumber
      & = \;  \sum_{i=1}^k \left( f_i(0)\,+\Delta_i N\right)       \\
        \label{eq.sumDeltasum}
       &   = \; \sum_{\ell=1}^N \ell \;+\;\left(\sum_{i=1}^k\Delta_i\right)N   \,.
\end{align}
Thus condition ($ii$) follows from the fact that $\sum_{i=1}^k\Delta_i=0$.  Finally, we know that $\sigma$
avoids $\deck$ because $\{\sigma(i)\}_{i\in \mathbb{Z}}$ can be partitioned into $k$ increasing subsequences. 

To prove the final statement of the lemma, let $\phi\in \E_N(\deck)$.
Partition $\phi$ into $k$ nonempty increasing periodic subsequences 
as in Proposition \ref{prop.unionincreasing}, writing 
$[N]=G_1\cup \cdots \cup G_k$ with $G_i=\{g_{i,j}:j\in [n_i]\}$
such that 
\begin{eqnarray*}  
 &  1\, \leq \,g_{i,1}\,<\,g_{i,2} \,< \,\ldots \,<\, g_{i,n_i} \, \leq \,N   \hspace{6mm}\hbox{and} & 
\\   & 
   \phi(g_{i,1}) \,<\,\phi(g_{i,2}) \,< \,\ldots \,<\,
   \phi(g_{i,n_i}) \, < \,\phi(g_{i,1}+N) \,=\, \phi(g_{i,1})+N\,. & 
\end{eqnarray*}
% the sequence $\phi(1),\ldots,\phi(N)$ 
% can be partitioned into $k$ nonempty increasing subsequences
%of some sizes $n_1,\ldots,n_k$. 
%For each $i \in[k]$,  write the $i^{th}$ subsequence as
%\[
%       \phi(g_{i,1}),\, \phi(g_{i,2}),  \, \ldots  ,\,\phi(g_{i,n_1})  \,   \hspace{5mm}  \hbox{with }
%    1\, \leq \,g_{i,1}\,<\,g_{i,2} \,< \,\ldots \,<\, g_{i,n_i} \, \leq \,N\,.
%\]
%Let $G_i\,=\, \{g_{i,j}: 1 \le j \le n_i\}$.  
Let $H_i$ be the $n_i$-element subset of 
$[1,N]$ consisting of the elements that are congruent mod $N$ to  $\{\phi(g_{i,j}):1 \le j \le n_i\}$.
Write the elements of $H_i$ as in Equation (\ref{eq.horder}).
Next, extend the definition of $h_{i,j}$ to all $j\in \mathbb{Z}$
by Equation (\ref{eq.hperiodic}).   Then $\phi(g_{i,1}) \,=\,h_{i,J}$ for some integer $J=J(i)$.  
Set $\Delta_i=J(i)-1$, so that 
$\phi=\Psi(\collec{n},\collec{G},\collec{H},\collec{\Delta})$.   
% [\textit{ *** IS THIS ADEQUATE? ****}] \justin{yes I think so} 
It remains only to show  that Equation (\ref{eq.Deltasum}) holds.
Since $\phi\in {\E}_N$, we know that $|h_{i,J}-g_{i,1}|<N$.
In particular, $h_{i,J}\,>\,-N+g_{i,1}\,\geq \,-N+1$.  Thus, by Equation (\ref{eq.hincr}) for $h$, 
we conclude that $J\geq 1-n_i$, i.e.\ that $\Delta_i\geq -n_i$.  Similarly, since $\phi(g_{i,0})=h_{i,J-1}$,
we have $|h_{i,J-1}-g_{i,0}|<N$ and $h_{i,J-1}<g_{i,0}+N\leq N$, and hence $J-1\leq n_i$, 
i.e.\ $\Delta_i\leq n_i$.   Thus $|\Delta_i|\leq n_i$ for each $i$.  Finally, the equation 
$\sum_{i=1}^k\Delta_i=0$ follows from Equation (\ref{eq.sumDeltasum}) and the 
fact that $\phi\in \E_N$.  
\end{proof}

\subsection{The asymptotic upper bound}
     \label{sec.upperbound}

%\begin{defi}
%   \label{def.Z}
%Let $n_1,\ldots,n_k$ be natural numbers.  Then we define
%$Z(n_1,\ldots,n_k)$ to be the number of $k$-tuples 
%$\collec{\Delta}=(\Delta_1, \ldots,\Delta_k)$ of integers such that 
%the conditions in Equation (\ref{eq.Deltasum}) hold. 
%\end{defi}

Recalling Definition \ref{def.domains}(b), it follows from the 
final assertion of Lemma \ref{lem.sigmaaffine} that
\begin{equation}
    \label{eq.4upbound}      
 |{\E}_N(\deck)|   \;\leq\;   \frac{1}{k!}  \sum_{\substack{n_1,\ldots,n_k\geq 1\\ n_1+\cdots+n_k=N}  }
      \left( \begin{array}{c}  N  \\ n_1,n_2,\cdots,n_k \end{array}  \right)^2 Z(n_1,\cdots,n_k).
\end{equation}
The division by $k!$ comes from the interchangeability of the the subscripts of $G_i$, $H_i$, $n_i$, and 
$\Delta_i$ (recall that each $n_i$ is non-zero). 
The basic idea behind the proof of Theorem \ref{thm.E4321} is to show
that this upper bound is asymptotically tight.  
% In marked contrast, we note that the analogous 
% upper bound (\ref{eq.S4321bd})  for $|{\cal S}_N(\deck)|$ is \textit{not} asymptotically tight.  

\medskip

The asymptotic behaviour of the sum of equation (\ref{eq.4upbound}) \textit{without} the $Z$ terms was established in 2009 by Richmond and Shallit  \cite{RS}:

\begin{thm}[\cite{RS}]
   \label{thm.richmond}
Fix an integer $k\geq 2$.   Then as $N\rightarrow \infty$,
\[      \sum_{\substack{n_1,\ldots,n_k\geq 0 \\ n_1+\ldots+n_k=N} } \binom{N}{n_1,n_2,\cdots,n_k}^2
    \;\sim  \;  k^{2N+k/2} \,(4\pi N)^{(1-k)/2}   \,.
\]
\end{thm}

The dominant terms of the sum of equation (\ref{eq.4upbound}) are those for which all $n_i$'s are
approximately equal.  This is quantified in the following result, which is a straightforward application 
of a well-known bound on tail probabilities.

\begin{lemma}
    \label{lem.alpha3}
Fix an integer $k\geq 2$.
 Fix $\alpha\in (0,1/k)$.  Then for every $N$ we have
 \begin{equation}
    \label{eq.Hoeff}
            \sum_{\substack{n_1,\ldots,n_k\geq 0\;: \\ \; n_1+\cdots+n_k=N, \\ \left|n_i-\frac{N}{k}\right| > \alpha N
  \; \textup{for some }i  }} 
     \binom{N}{n_1,\cdots,n_k}^2   \;\leq \;  4\,k^{2N+2} e^{-4N\alpha^2}  \,.   
\end{equation}
In particular,  as $N\rightarrow \infty$,
 \begin{equation}
    \label{eq.Hoeff2}
        \sum_{\substack{n_1,\ldots,n_k\geq 1 \;: \\ \; n_1+\cdots+n_k=N, \\ \left|n_i-\frac{N}{k}\right| \leq \alpha N
  \; \textup{for all }i  }} 
     \binom{N}{n_1,\ldots,n_k}^2 
      \;\sim \;   
      \sum_{\substack{n_1,\ldots,n_k\geq 1\;: \\ \; n_1+\cdots+n_k=N} } \binom{N}{n_1,\ldots,n_k}^2 \,.
 \end{equation}
 \end{lemma}

\begin{proof}
Consider a sequence of $N$ independent random variables $X_1,\ldots,X_N$, where each $X_j$ is
chosen uniformly at random from the set $\{1,\ldots,k\}$.  For $i \in [k]$, let $Y_i$ be the 
number of $X_j$'s that are equal to $i$.  Then the joint distribution of $(Y_1,\ldots,Y_k)$ is 
multinomial with parameters $N$ and $p_1=\cdots=p_k=1/k$.  Also, the (marginal) distribution of each
$Y_i$ is binomial with parameters $N$ and $p=1/k$.
Thus we have
\begin{align*}
  \sum_{\substack{n_1,\ldots,n_k\;: \\ \; n_1+\cdots+n_k=N, \\ \left|n_i-\frac{N}{k}\right| > \alpha N
  \; \textup{for some }i  }} 
    \!\!\! \binom{N}{n_1,\cdots,n_k} \,k^{-N}  \; & =  \;
     \Pr\left(\left|Y_i-\frac{N}{k}\right|>\alpha N  \hbox{ for some }i\right)
   \\
   & \leq \;\sum_{i=1}^k   \Pr\left(\left|\frac{Y_i}{N}-\frac{1}{k}\right|>\alpha \right)
   \\
   & \leq 2\,k\, e^{-2N\alpha^2}   \,,
\end{align*}
where the last line uses Hoeffding's Inequality applied to the binomial distribution (Theorem 2 of \cite{Ho}).
The inequality (\ref{eq.Hoeff}) follows directly.

The inequality (\ref{eq.Hoeff2}) follows from (\ref{eq.Hoeff}) and Theorem \ref{thm.richmond}.
\end{proof}

We shall also need to understand the asymptotics of $Z(n_1,\cdots,n_k)$ when each $n_i$ is 
close to $N/k$.
We remark that $Z(n_1,\ldots,n_k)$ is the coefficient of $x^N$ (the middle term) in
$\prod_{i=1}^k (1+x+x^2+\cdots+x^{2n_i})$.   This connection was made 
in 1876 by D\'{e}sir\'{e} Andr\'{e}, when he proved the following result.

\begin{thm}
    \label{thm.andre}
For the case $n_i=n$ for every $i=1,\ldots,k$, we have 
\[    Z(n,\ldots,n)    \;=\;  k \sum_{j=0}^{\lfloor kn/(2n+1)\rfloor}  (-1)^j   
      \frac{(k+kn-j(2n+1)-1)!}{j!(k-j)!(kn-j(2n+1))!}   \,.
\]
\end{thm}

\begin{proof} This result is given in Remarks 60 and 61 of Andr\'{e} \cite{An}.  In the notation 
in that paper,  $Z(n,\ldots,n)  = (k,kn)_{2n}$.  The same result is presented with a different proof 
in \cite{CR}.
\end{proof}

\begin{cor}
   \label{cor.andre1}
For the case $n_i=n$ for every $i=1,\ldots,k$, we have   
\[   \lim_{n\rightarrow\infty}   \frac{Z(n,\ldots,n)}{n^{k-1}}   \;=\;  
        \frac{1}{(k-1)!}\sum_{j=0}^{\lfloor k/2 \rfloor}  (-1)^j  \binom{k}{j}\left(  k-2j\right)^{k-1} \;=\; Z^*_k.
\]
\end{cor}
\noindent 
Recall that $Z^*_k$ was defined in Equation (\ref{eq.Zstardef}).

\begin{proof}  This follows directly from Theorem \ref{thm.andre}.  Notice that when $k$ is even, 
the summand for $j=k/2$ is 0.
\end{proof}

%\begin{rem}  We note that $Z^*_k/2^{k-1}$ is the value of the probability density function
%of the sum of $k$ independent uniform random variables on 
%$[0,1]$ evaluated at its midpoint, $k/2$; see ``Irwin--Hall distribution'' on Wikipedia.
%We easily compute $Z_2^*=2$, $Z_3^*=3$, $Z_4^*=16/3$, and $Z_5^*=115/12$.
%\end{rem}

We are now ready to prove the asymptotic upper bound corresponding to Theorem \ref{thm.E4321}.

\begin{prop}
   \label{prop.4321upper}
\[    \limsup_{N\rightarrow\infty}  
     \frac{|  {\E}_N(\deck)|}{k^{2N}N^{(k-1)/2} }  \;\leq \; 
       \frac{Z^*_k}{k^{k/2}  \,(4\pi)^{(k-1)/2}\,(k-1)!} \,.
\]
\end{prop}

\begin{proof}
Let $\alpha\in (0,1/k)$.   The upper bound (\ref{eq.4upbound}) says that 
\[    
       |  {\E}_N(\deck)|   \;   \leq \;   \frac{1}{k!} \left({\sum}_{N,\leq}  \,+\,  {\sum}_{N,>} \right)  \,,  
\]
where 
\begin{align*}
    {\sum}_{N,\leq}  \; & = \;     
      \sum_{\substack{n_1,\ldots,n_k \ge 1\;: \\ \; n_1+\cdots+n_k=N, \\ \left|n_i-\frac{N}{k}\right| \leq \alpha N
  \; \textup{for all }i  }} 
     \binom{N}{n_1,\ldots,n_k}^2  Z(n_1,\ldots,n_k)    \hspace{5mm}  \hbox{and}
     \\
   {\sum}_{N,>}  \; & = \; 
   \sum_{\substack{n_1,\ldots,n_k \ge 1\;: \\ \; n_1+\cdots+n_k=N, \\ \left|n_i-\frac{N}{k}\right| > \alpha N
  \; \textup{for some }i  }} 
     \binom{N}{n_1,\cdots,n_k}^2  Z(n_1,\ldots,n_k)  \,.
\end{align*}
In view of the obvious bound $Z(n_1,\ldots,n_k)\leq (2N+1)^{k-1}$, we see from inequality 
(\ref{eq.Hoeff}) that ${\sum}_{N,>} \,=\, o(k^{2N})$ 
as $N\rightarrow\infty$.
We can now turn to the asymptotics of ${\sum}_{N,\leq}$. 
%\justin{I believe that this alone does not imply that the ``$\le$'' sum determines the asymptotic behavior, without already knowing that the ``$\le$'' sum is big enough.}
% Therefore the sum
% ${\sum}_{N,\leq}$ will determine the asymptotic behaviour of $|{\E}_N(\deck)$.

Let $u_N = \left\lfloor \frac{N}{k}+\alpha N\right\rfloor$.  In every term of ${\sum}_{N,\leq}$
we know that $n_i\leq u_N$ for each $i$, and hence $Z(n_1,\ldots,n_k)\leq Z(u_N,\ldots,u_N)$
(since obviously $Z$ is non-decreasing in each $n_i$).  Thus we have 
\begin{equation}
    \label{eq.sumupbd1}
     {\sum}_{N,\leq}  \;\leq \;   Z(u_N,\ldots,u_N)   
     \sum_{\substack{n_1,\ldots,n_k\geq 0\;: \\ \; n_1+\ldots+n_k=N} } \binom{N}{n_1,n_2,\cdots,n_k}^2 \,.
\end{equation}
By Corollary \ref{cor.andre1}, we have $Z(u_N,\ldots,u_N)\,\sim\, (u_N)^{k-1}Z^*_k$ 
as $N\rightarrow\infty$.  Combining this with Theorem \ref{thm.richmond}, we obtain
\[    
   {\sum}_{N,\leq}  \;\leq \; N^{k-1}\left(\frac{1}{k}+\alpha\right)^{k-1}Z^*_k \times
       k^{2N+k/2} \,(4\pi N)^{(1-k)/2}  \times (1+o(1)) \,.
\]
Since ${\sum}_{N,>} \,=\, o(k^{2N})$, it follows that
\[
   \limsup_{N\rightarrow\infty}  
     \frac{|  {\E}_N(\deck)|}{k^{2N}N^{(k-1)/2} }  \;\leq \; \frac{1}{k!}
     \,\left(\frac{1}{k}+\alpha\right)^{k-1}Z^*_k  
     \,k^{k/2}\,  \left( \frac{1}{4\pi}\right)^{(k-1)/2}  \,.  
\]
Since the positive number $\alpha$ can be made arbitrarily close to 0, the proposition follows.
\end{proof}

\subsection{The asymptotic lower bound}
   \label{sec.lowerbound}

We start with some notation.
For positive integers $w$ and $N$, recall that $\Seq(w;N)$ is the set
of all $w$-element subsets of $[N]$.  In this section, we shall 
write a member of $\Seq(w;N)$ as a $w$-element vector with
the entries in increasing order:  $\vec{x}=(x(1),x(2),\cdots,x(w))$, 
with $x(1)<\cdots < x(w)$.  
% (as we did for $G_i$ and $H_i$ above).

In applying the following definition, we shall want $\alpha$ to 
be small, and $A$ and $B$ to be large.

\begin{defi}
    \label{def.lower}
Fix $k\geq 2$.
Let $w,N\in \mathbb{N}$, and let $\alpha$, $A$, and $B$ 
be positive real numbers. 
\\    
(a) Define 
\[    \Seq^{*A}(w;N)  \;=\;  \left\{\vec{x}\in \Seq(w;N):
     \left|x(\ell)-\ell\,\frac{N}{w+1}\right|<A \hbox{ for all }\ell \in [w]
     \right\} 
\]
(Roughly speaking, a $w$-element subset of $[N]$ is in
$\Seq^{*A}(w;N)$ if its elements are within distance $A$ of 
a uniform spacing configuration over the interval $[0,N]$.) 
\\
(b) For $\collec{n}\in \mathbb{N}^k$ such that $n_1+\cdots+n_k=N$,
define
\begin{align*}   
   {\cal V}^{**A}_N(\collec{n})   \;=\; &  \left\{  (G_1,\ldots,G_k): \, \{G_1,\ldots,G_k\}   \mbox{ is a partition of }\{1,\ldots,N\}
       \right.
        \\ & \hspace{5mm}
         \left.  \text{with } G_i\in \Seq^{*A}(n_i;N) \text{ for each } i \in [k]  \right\}.  
\end{align*}
(c) Let $\mathfrak{N}(N,\alpha)$ be the set of 
$\collec{n}\in \mathbb{N}^k$ such that $n_1+\cdots+n_k=N$ and 
\begin{equation} \label{eqn:defNNa}  \left|n_i-\frac{N}{k}\right| \leq \alpha N
  \hspace{5mm}\hbox{for each } i\in [k] \,.
\end{equation}
(d) Let $\mathcal{D}_1(N,\alpha,A,B)$  be the set of all 
$(\collec{n},\collec{G},\collec{H},\collec{\Delta}) \in
\mathcal{D}_0(N)$ with the additional constraints that 
$\collec{n}\in \mathfrak{N}(N,\alpha)$, $|\Delta_i|<n_i-B$ for every
$i\in [k]$, and 
$\collec{G},\collec{H}\in {\cal V}^{**A}(\collec{n})$.
\end{defi}
  
The following lemma establishes some regularity properties of the 
(mildly) reduced domain $\mathcal{D}_1(N,\alpha,A,B)$ of $\Psi$.
  
\begin{lemma}
   \label{lem.V**}
Let $k\geq 2$.  Let $N\in \mathbb{N}$, let $A,B>0$, and let $\alpha\in(0,1/k)$.  
%Let $n_1,\ldots,n_k$  be natural numbers and let $N = n_1 +\cdots + n_k$.
%Assume the following: 
%\begin{itemize}
%\item $\left|n_i-\frac{N}{k}\right| \leq \alpha N$ for each $i\in [k]$;
%\item $\collec{G}$ and $\collec{H}$ are in ${\cal V}^{**A}_N(n_1,\ldots,n_k)$; and
%\item  $\Delta_1,\ldots,\Delta_k$ are integers whose sum is $0$ and that satisfy
%$|\Delta_i|<n_i-B$ for all $i\in[k]$.
%\end{itemize}
Let $\vec{v}=(\collec{n},\collec{G},\collec{H},\collec{\Delta})
\in \mathcal{D}_1(N,\alpha,A,B)$.
Let $\sigma=\Psi(\vec{v})$ be the affine permutation defined 
as in Equation (\ref{eq.sigma4def}).
% Lemma \ref{lem.sigmaaffine}.
% by Equations (\ref{eq.hperiodic}) and (\ref{eq.sigma4def}).
Then for every $i\in[k]$ and $j\in \mathbb{Z}$, 
\begin{align}
    \label{eq.stripbound1}
     \left| g_{i,j}\,-\, j \frac{N}{n_i}\right| \; & <  \; 
     A+\frac{k}{1-k\alpha} \,,
    \\
    \label{eq.stripbound2}
     \left| h_{i,j}\,-\, j \frac{N}{n_i}\right| \; & <  \; 
     A+\frac{k}{1-k\alpha} \,,
    \\
    \label{eq.stripbound3}
    \left| \sigma(g_{i,j}) \,-\, \left(g_{i,j}\,+\, \Delta_i \frac{N}{n_i}\right)\right| \; & <  \; 
        2\left(A+\frac{k}{1-k\alpha}\right) \,,
    \\
    \hbox{and}   \hspace{11mm}
 \left| \sigma(g_{i,j}) \,-\, g_{i,j} \right|    \; & < \;   N\,-\, \frac{kB}{1+k\alpha}  \,+\,   2A+\frac{2k}{1-k\alpha} \,.
       \label{eq.stripbound4}
\end{align}
In particular,  $\sigma$ is a bounded affine permutation if 
$\frac{kB}{1+k\alpha}  \,\geq \,   2A+\frac{2k}{1-k\alpha}$.
\end{lemma}
  
\begin{proof}
First we observe that for every $\ell \in [n_i]$, 
\begin{equation}
    \label{eq.ellNbound}
       \left|   \ell \frac{N}{n_i+1}  \,-\, \ell \frac{N}{n_i}\right|   \;=\;   \frac{\ell N}{n_i(n_i+1)}  \;<\; \frac{N}{n_i}
           \;\leq \;  \frac{N}{\frac{N}{k}-\alpha N}    \;=\;  \frac{k}{1-k\alpha} \,.
\end{equation}
This bound and the definition of ${\cal V}^{**A}_N(\collec{n})$ 
imply Equations (\ref{eq.stripbound1}) and
(\ref{eq.stripbound2}) for $j\in [n_i]$, and the extension to all $j\in \mathbb{Z}$ follows from 
Equation (\ref{eq.hperiodic}).  Equation (\ref{eq.stripbound3}) follows from Equations (\ref{eq.sigma4def}),
(\ref{eq.stripbound1}), and (\ref{eq.stripbound2}).   Equation (\ref{eq.stripbound4}) follows from
Equation (\ref{eq.stripbound3}) and
\[    \left| \Delta_i \frac{N}{n_i}\right|   \;\leq \;  \frac{(n_i-B)N}{n_i}   \;\leq \; 
    \left(1-\frac{B}{\frac{N}{k}+\alpha N}\right)N     \;=\;  N\,-\, \frac{kB}{1+k\alpha}  
\]
(using Definition \ref{def.lower}(c,d)).
The final assertion of the lemma is a consequence of Equation (\ref{eq.stripbound4}).
\end{proof}

We shall now define the restricted domain $\Dom$ on which
$\Psi$ is $k!$-to-one.

\begin{defi}
   \label{def.Dom}
Let $N\in \mathbb{N}$, let $A,B>0$, and let $\alpha\in(0,1/k)$.
Let $\Dom = \Dom(N,\alpha,A,B)$ be the set of $(4k)$-tuples
$(\collec{n},\collec{G},\collec{H},\collec{\Delta})$ 
in $\mathcal{D}_1(N,\alpha,A,B)$  that also satisfy 
%that  satisfy the assumptions of Lemma \ref{lem.V**} as well as 
\begin{equation}
   \label{eq.Deltadiff}
    \left|   \frac{\Delta_iN}{n_i}  \,-\,  \frac{\Delta_{i'}N}{n_{i'}}\right|   \;>\;  4\left(2A+\frac{2k}{1-k\alpha}\right)
   \hspace{5mm}\hbox{whenever $i,i'\in [k]$ and $i\neq i'$}.
\end{equation}
\end{defi}

\medskip

\begin{lemma}
    \label{lem.V**strips}
Let $N\in \mathbb{N}$, $A,B>0$, and let $\alpha\in(0,1/k)$.
%be the set of $(4k)$-tuples
%$(\collec{n},\collec{G},\collec{H},\collec{\Delta})$ that  
%satisfy the assumptions of Lemma \ref{lem.V**} as well as 
%\begin{equation}
%   \label{eq.Deltadiff}
%    |\Delta_i -\Delta_{i'}|    \;>\;  2\alpha N   \,+\, 4A  \,+\, \frac{8}{1-k\alpha}  
%        \hspace{5mm}\hbox{whenever $i,i'\in [k]$ and $i\neq i'$}.
%    \label{eq.Deltadiff2}
%    \left|   \frac{\Delta_iN}{n_i}  \,-\,   \frac{\Delta_{i'}N}{n_{i'}}\right|   \;>\;  4\left(2A+\frac{2k}{1-k\alpha}\right)
%   \hspace{5mm}\hbox{whenever $i,i'\in [k]$ and $i\neq i'$}.
%\end{equation}  
Then the restriction of the function $\Psi$ to $\Dom$ is exactly $k!$-to-1.
\end{lemma}

\begin{proof}
Let
$\vec{v}=(\collec{n},\collec{G},\collec{H},\collec{\Delta})\in \Dom$
and let $\sigma=\Psi(\vec{v})$.  
We also define the (truncated) plot of $\sigma$ to be 
% [MOVE THIS EARLIER?] \justin{I think it's fine here}
\[        \textbf{Plot}[\sigma]  \;:=\;   \{ (i,\sigma(i))\,:\, i\in[N]\}  \,.   \]

%$\sigma = \Psi(G_1,G_2,G_3,H_1,H_2,H_3,\Delta_1,\Delta_2,\Delta_3)$.
%Assume also that 
% $\sigma = \Psi(G^*_1,G^*_2,G^*_3,H^*_1,H^*_2,H^*_3,\Delta^*_1,\Delta^*_2,\Delta^*_3)$,

For each real $b$, define 
\[      \textbf{Strip}[b]   \;:=  \;    \left\{ (x,y)\in \mathbb{R}^2:  |y-(x+b)| \,<\, 2A+\frac{2k}{1-k\alpha} \right\} \,,
\]
a diagonal strip shifted vertically by $b$.
For each $i\in[k]$, Equation (\ref{eq.stripbound3}) of Lemma \ref{lem.V**} tells us that
the points $(g_{i,j},\sigma(g_{i,j}))$ ($j=1,\ldots,n_i$) are all in $\textbf{Strip}[\Delta_iN/n_i]$.  Hence
\[     \textbf{Plot}[\sigma] \;\subseteq \;   \bigcup_{i=1}^k  \textbf{Strip}\left[\frac{\Delta_iN}{n_i}\right]  \,.
\]
The purpose of the condition (\ref{eq.Deltadiff}) is to ensure that the $k$ strips $\textbf{Strip}[\Delta_iN/n_i]$
are not only disjoint but also are
separated by at least the width of a strip.
%far enough apart so that for any choice of $b_1$, $b_2$, and $b_3$ such 
%that each intersection $\textbf{Plot}[\sigma]\cap \textbf{Strip}[b_i]$ is nonempty  [*** NO--- REWRITE THIS***], these three 
%intersections determine $\vec{v}$ up to a permutation of the subscripts.

%Suppose we know the inequalities 
%\begin{equation}
%    \label{eq.Deltadiff2}
%    \left|   \frac{\Delta_iN}{n_i}  \,-\,  \frac{\Delta_{i'}N}{n_{i'}}\right|   \;>\;  4\left(2A+\frac{2k}{1-k\alpha}\right)
%   \hspace{5mm}\hbox{whenever $i,i'\in [k]$ and $i\neq i'$}.
%\end{equation}  
%If (\ref{eq.Deltadiff2}) holds, 
For any real $b$, 
the strip $\textbf{Strip}[b]$ cannot intersect more than one of 
the $k$ strips $\textbf{Strip}[\Delta_iN/n_i]$.  Therefore for any choice of $b_1,\ldots,b_k$ such that
 $\textbf{Plot}[\sigma]$ is contained in $\cup_{i=1}^k \textbf{Strip}[b_i]$, the partition of
 the $N$ points of  $\textbf{Plot}[\sigma]$ into the $k$ parts $\textbf{Plot}[\sigma]\cap \textbf{Strip}[b_i]$
 ($i\in [k]$)  must be the same partition as the one given by $\textbf{Plot}[\sigma]\cap \textbf{Strip}[\Delta_iN/n_i]$ 
 ($i\in[k]$), up to permutation of the $k$ parts.   This partition determines  $G_i$ (the first coordinates
 of the points in the $i^{th}$ part) and $H_i$ (the second coordinates of the points, modulo $N$).
 Finally, $\Delta_i$ is determined by Equation (\ref{eq.sigma4def}). Thus the lemma is proved. \end{proof}
 %subject to the verification of (\ref{eq.Deltadiff2}).
% 
% It remains to prove the inequalities of (\ref{eq.Deltadiff2}).  
% Observe first that by the assumptions of Lemma \ref{lem.V**}, we have
%\[       \left|\frac{\Delta_i}{n_i}\left(N-kn_i\right) \right|   \;\leq \; k\alpha N    
%     \hspace{5mm}  \hbox{and}\hspace{5mm}
%       \left|\frac{\Delta_{i'}}{n_{i'}}\left(N-kn_{i'}\right) \right|   \;\leq \; k \alpha N  \,.
%\]
%Therefore       
%\begin{align*}
%      \left|   \frac{\Delta_iN}{n_i}  \,-\,  \frac{\Delta_i'N}{n_{i'}}\right|  \;  & = \;
%      \left| k(\Delta_i-\Delta_{i'})   \,+\, \frac{\Delta_i}{n_i}\left(N-kn_i\right) 
%         \,+\, \frac{\Delta_{i'}}{n_{i'}}\left(N-kn_{i'}\right) \right| 
%              \\
%        & \geq \;    k|\Delta_i-\Delta_{i'}| \,-\,  2k\alpha N  \\
%        & >   \;2k\alpha N  +4Ak  +\frac{8k}{1-k\alpha} - 2k\alpha N  
%          \hspace{8mm}\hbox{(by Equation (\ref{eq.Deltadiff})). }
%\end{align*}
%The desired inequalities follow directly from the above.  % and Equation (\ref{eq.Deltadiff}). 
% \hfill   $\Box$

\begin{cor}
   \label{cor.Dombound}
Let  $N\in \mathbb{N}$, $A,B>0$, and let $\alpha\in(0,1/k)$.     
Assume $\frac{kB}{1+k\alpha}  \,\geq \,   2A+\frac{2k}{1-k\alpha}$.  Then the function
$\Psi$ maps  $\Dom(N,\alpha,A,B)$ into $\E_N(\deck)$, and
\[    |{\E}_N(\deck)|  \;\geq \; \frac{1}{k!}\,| \Dom(N,\alpha,A,B)| \,.
\]
\end{cor}

\noindent
\begin{proof}  The first assertion follows from
Lemma \ref{lem.sigmaaffine} and
the last sentence in the statement of Lemma \ref{lem.V**}. 
%the function
%$\Psi$ maps  $\Dom(N,\alpha,A,B)$ into $\E_N(\deck)$.     
The second assertion follows from Lemma \ref{lem.V**strips}.
\end{proof}

Our job now is to estimate the size of $\Dom(N,\alpha,A,B)$.  

\begin{lemma}
   \label{lem.W**}
Fix $k\geq 2$.
Fix $A>0$, $B>0$, and $\alpha\in(0,1/k)$.
%Assume $n_1+\cdots+n_k=N$ and $\left|n_i-\frac{N}{k}\right| \leq  \alpha N$ for each $i$.  
For       % $\collec{n}\in \mathfrak{N}(N,\alpha)$, 
natural numbers $n_1,\ldots,n_k$, 
let $\mathcal{W}(\alpha,A,B,\collec{n})$
be the set of ordered $k$-tuples $(\Delta_1,\cdots,\Delta_k)$ of integers whose sum is 0
and which satisfy Equation (\ref{eq.Deltadiff}) (with $N=n_1+\cdots+n_k$) 
as well as $|\Delta_i|\leq n_i-B$ for each $i$.  
Then we have 
\begin{equation}
    \label{eq.W**0}
      |\Dom(N,\alpha,A,B)|   \;= \;    
   %   \sum_{\substack{n_1,\cdots,n_k\;: \\ \; n_1+\cdots+n_k=N, \\ \left|n_i-\frac{N}{k}\right| \leq \alpha N \;\forall i  }} 
   \sum_{\collec{n}\in \mathfrak{N}(N,\alpha)}
  |\mathcal{V}^{**A}(\collec{n})|^2\, | \mathcal{W}(\alpha,A,B,\collec{n})| \,.
\end{equation}
Moreover, let $t_N=\left\lfloor \frac{N}{k}-\alpha N-B\right\rfloor$
and $\Theta_N= 
%2\alphaa N+
8A+8k/(1-k\alpha)$ (which is the right-hand side of Equation (\ref{eq.Deltadiff})). 
%If also $\left|n_i-\frac{N}{k}\right|\leq
%\alpha N$ for every $i\in[k]$, then 
Then for every $\collec{n}\in\mathfrak{N}(N,\alpha)$, we have
\begin{equation}
   \label{eq.W**1}
     | \mathcal{W}(\alpha,A,B,\collec{n})|  \;\geq \;  Z(t_N,\cdots,t_N)\,-\, \binom{k}{2}(2N)^{k-2}(2\Theta_N+1)\,.
\end{equation}
\end{lemma}
\begin{proof}
Equation (\ref{eq.W**0}) follows from the definition of $\Dom$.

Let $\mathcal{W}^-$ be the set of ordered $k$-tuples
$\collec{\Delta}$ of integers whose 
sum is 0 and satisfy $|\Delta_i|\leq n_i-B$ for each $i$. 
By our assumptions, we have $n_i-B\geq t_N$ for each $i$, and hence 
$|\mathcal{W}^-|\,\geq \,Z(t_N,\ldots,t_N)$ (since $Z$ is nondecreasing in each argument). 
Now, for each two-element subset $\{i,i'\}$ of $[k]$, the number of $k$-tuples $\collec{\Delta}$ in $\mathcal{W}^-$
that violate Equation (\ref{eq.Deltadiff}) is at most $(2n_i)(2\Theta_N+1)(2N)^{k-3}$
(first choose $\Delta_i$, then $\Delta_{i'}$, then $\Delta_j$ for $k-3$ of the remaining indices $j$ in $[k]$;
the final $\Delta_j$ is determined because $\sum_j \Delta_j = 0$).   
Equation (\ref{eq.W**1}) follows.
\end{proof}

The main task that remains is to get a lower bound on $|\mathcal{V}^{**A}(\collec{n})|$.  This is
accomplished by the following lemma. 
It is an adaptation of part of Lemma 21 in \cite{MP}.
%Madras and Pehlivan (EJC, vol 23, Paper \#P4.36, 2016).

\medskip

\begin{lemma}
   \label{lem.sequnifB}
Fix $k\geq 2$. Let $A>0$ and let $\alpha\in(0,1/k)$.  Then there exist positive constants $C(\alpha)$ and $\tilde{N}(\alpha)$
such that 
\begin{equation*}
   %  \frac{|\Seq^{*A}(w;M)|}{\binom{M}{w} } 
   \frac{|\mathcal{V}^{**A}(n_1,\ldots,n_k)|}{ \binom{N}{n_1,\cdots,n_k} }
      \;\geq \; \; 1-\frac{C(\alpha)N^{3/2}}{A^2}  
     \hspace{5mm} \hbox{ whenever } N\,\geq \,\tilde{N}(\alpha) 
     \hbox{ and }\collec{n}\in \mathfrak{N}(N,\alpha).
%     , \\
%   \quad   n_1+\cdots+n_k\,=\,N, &
%     \hspace{3mm}\hbox{ and }  \hspace{3mm}
%   \frac{1}{k}-\alpha_1  \;\leq \; \frac{n_i}{N}  \;\leq \;  \frac{1}{k}+\alpha_1 \quad\hbox{for each $i\in[k]$}\,.
\end{equation*}
\end{lemma}

\begin{proof}
We shall prove the lemma by converting it into a probabilistic statement.
Fix $N$, and choose $\collec{n}\in \mathfrak{N}(N,\alpha)$.
% $n_1,\ldots,n_k$ as specified in the statement  of the lemma.
Now, choose $(\textbf{G}_1,\ldots,\textbf{G}_k)$ uniformly at random 
from the collection of all $\binom{N}{n_1,\cdots,n_k}$ 
partitions of $\{1,\ldots,N\}$ for which the $i^{th}$ part has size $n_i$.
% It is not hard to see that for 
For each $i$, by symmetry, the random set $\textbf{G}_i$ is 
uniformly distributed on the 
collection of all $n_i$-element subsets of $\{1,\ldots,N\}$.
It follows that
\begin{align*}
   %  \frac{|\Seq^{*A}(w;M)|}{\binom{M}{w} } 
  1\;-\; \frac{|\mathcal{V}^{**A}(n_1,\ldots,n_k)|}{
    \binom{N}{n_1,\cdots,n_k} }
    \; & = \;   \Pr\left(\textbf{G}_i\not\in \Seq^{*A}(n_i,N) \hbox{ for some }i\right) 
       \\
    & \leq \;   \sum_{i=1}^k   \Pr\left(\textbf{G}_i\not\in \Seq^{*A}(n_i,N) \right)    
    \\
    & = \;   \sum_{i=1}^k  \left( 1\,-\,  \frac{|\Seq^{*A}(n_i;N)|}{\binom{N}{n_i} } \right)\,.
\end{align*}
We shall complete the proof by deriving an upper bound on $1-|\Seq^{*A}(w;N)|/\binom{N}{w}$, assuming that
$\frac{1}{k}-\alpha\leq \frac{w}{N}\leq \frac{1}{k}+\alpha$
(which is satisfied for $w=n_i$, since $\collec{n}\in \mathfrak{N}(N,\alpha)$).

Let $p\in (0,1)$.  Let $X_1,X_2,\ldots$ be a sequence of independent
random variables having the  geometric distribution with parameter $p$;
that is, $\Pr(X_i=\ell)=p(1-p)^{\ell-1}$ for $\ell=1,2,\ldots$.
Next, let $T_i=X_1+X_2+\cdots+X_i$ for each $i$.
These random variables have negative binomial distributions
\begin{equation}
   \label{eq.negbinB}
    \Pr(T_{j+1}=\ell+1)  \;=\;  \binom{\ell}{j}p^{j+1}(1-p)^{\ell-j}  
    \hspace{5mm}\hbox{for $\ell \geq j$}.
\end{equation}
Moreover, for any $\vec{x}\in \Seq(w;N)$ (writing $x(0)=0$ and $x(w+1)=N+1$),
\begin{align}
   \nonumber
   \Pr(T_{\ell}=x(\ell) \hbox{ for }\ell=1,\ldots,w \,|\, T_{w+1}=N{+}1)
    \; & =\;  \frac{ \prod_{\ell=1}^{w+1} p(1-p)^{x(\ell)-x(\ell-1)-1} }{
         \binom{N}{w}p^{w+1}(1-p)^{N-w} }
   \\  \label{eq.geomjointB}
     & = \;   \binom{N}{w}^{-1}  \,.
\end{align}  
Equation (\ref{eq.geomjointB}) says that
\textit{the conditional distribution of $(T_1,\ldots,T_w)$, given 
that $T_{w+1}=N+1$, is precisely the uniform distribution on $\Seq(w;N)$.}
This assertion is true for any $p$.  Let us now fix $p=(w+1)/N$; we shall
soon see why this is a convenient choice.

By Equation (\ref{eq.geomjointB}),  
\begin{equation}
    \nonumber
   \frac{ |\Seq^{*A}(w;N)|}{\binom{N}{w}}  \; = \;
     \Pr\left( |T_{\ell}-\ell/p| < A \hbox{ for }l=1,\ldots,w \,|\, 
       T_{w+1}=N+1  \right)   \,.
\end{equation}
and therefore
\begin{equation}
    \label{eq.ratioprobB}
   0\;\leq \; 1 \,-\, \frac{ |\Seq^{*A}(w;N)|}{\binom{N}{w}}  
   \; \leq  \; \frac{ \Pr(\max_{\ell=1,\ldots,w}|T_{\ell}-\ell/p| \geq  A )}{
        \Pr(T_{w+1}=N+1)} .
\end{equation} 

From Stirling's Formula $m!\sim \sqrt{2\pi m}\,(m/e)^m$, we see that there is a 
constant $C_s>0$  such that 
\[     \frac{1}{C_s}\,\frac{m^{m+1/2}}{e^m}  \;\leq \;   m!  \;\leq\;    C_s\, \frac{m^{m+1/2}}{e^m} 
    \hspace{5mm}\hbox{for every positive integer $m$.}
\]
It follows from these bounds and Equation (\ref{eq.negbinB}) that
\begin{align}
    \nonumber
   \Pr(T_{w+1}= & N+1)  \; =\;
     \frac{N!}{w!(N-w)!}  \frac{(w+1)^{w+1}(N-w-1)^{N-w}}{N^{N+1} }    \\
    \nonumber
    & \geq  \;  \frac{N^{N+1/2}}{ C_s^{3} w^{w+1/2}(N-w)^{N-w+1/2}}
      \frac{(w+1)^{w+1}(N-w-1)^{N-w}}{N^{N+1} } 
              \\
   &  \geq \;   \frac{\sqrt{w}}{C_s^3 \sqrt{N}\sqrt{N-w}} \left(\frac{w+1}{w}\right)^{w+1}
          \left(1-\frac{1}{N-w}\right)^{N-w}\,.
\label{eq.denom0B}
\end{align}
By calculus, one can show that $\left(1-\frac{1}{t}\right)^t \geq \frac{1}{4}$ whenever $t\geq 2$.
Therefore,  we conclude from (\ref{eq.denom0B}) that
\begin{align}
    \Pr(T_{w+1}=  N+1)        \; \geq \; & \frac{1}{C_s^3}  \,
    \sqrt{ \frac{ \frac{1}{k}-\alpha}{
  \left(1-\frac{1}{k}+\alpha\right)N}} \times 1 \times \frac{1}{4}  
         \label{eq.denomB}
         \\
         \nonumber
        & \hspace{5mm}\hbox{if $N-w\geq 2$ and 
        $\frac{1}{k}-\alpha  \;\leq \; \frac{w}{N}  \;\leq \;  \frac{1}{k}+\alpha$}. 
\end{align}
Observe that under the constraints on $w/N$,
the condition $N\geq 2/(1-\frac{1}{k}-\alpha)$ implies $N-w\geq 2$.

Since the random 
variables $X_i$ have mean $1/p$ and variance $(1-p)/p^2$, we also have
\[    E(T_{\ell}) \;=\; \frac{\ell}{p}    \hspace{5mm} \hbox{and}\hspace{5mm}
       \textup{Var}(T_{\ell})  \;=\;   \frac{ \ell (1-p)}{p^2} \,.   
\]
For the numerator of the right-hand side of Equation (\ref{eq.ratioprobB}),
we use Kolmogorov's Inequality (see for example section IX.7 of \cite{Fe}), 
which may be viewed as a strengthening
of Chebychev's Inequality that is applicable to sums of independent random variables.
\begin{align}
   \nonumber
  \Pr\left(\max_{\ell=1,\ldots,w}|T_{\ell}-\ell/p| \geq  A \right)
    \;& \leq \;  \frac{\textup{Var}(T_w)}{A^2}   \;=\;  \frac{w(1-p)/p^2}{A^2} \\
     & \leq   \;   \frac{w N^2}{(w+1)^2A^2}   \;<\;   \frac{N^2}{wA^2}   \nonumber  \\
     & \leq \;  \frac{N}{\left(\frac{1}{k}-\alpha\right)A^2} \,.
    \label{eq.numerB}
\end{align}
Applying  Equations (\ref{eq.denomB}) and (\ref{eq.numerB}) 
to Equation (\ref{eq.ratioprobB})  shows that 
\[    1 \,-\, \frac{ |\Seq^{*A}(w;N)|}{\binom{N}{w}}  \;\leq \;   
  \frac{4C_s^3\,N^{3/2}}{ \left(\frac{1}{k}-\alpha\right)A^2}
    \sqrt{  \frac{ 1-\frac{1}{k}+\alpha}{\frac{1}{k}-\alpha} } \,.
\]
Taking $\tilde{N}(\alpha)=2/\left(1-\frac{1}{k}-\alpha\right)$ and 
$C(\alpha)= 4kC_s^3\sqrt{\left(1-\frac{1}{k}+\alpha\right) \left(\frac{1}{k}-\alpha\right)^{-3}}$ , 
the proof of Lemma \ref{lem.sequnifB} is now complete.
\end{proof}

We can now complete the proof of Theorem \ref{thm.E4321}.

\begin{prop}
   \label{prop.4321lower}
Fix $k\geq 2$.  Then   
\[    \liminf_{N\rightarrow\infty}  
     \frac{|  {\E}_N(\deck)|}{k^{2N}N^{(k-1)/2} }  \;\geq \; 
       \frac{Z^*_k}{k^{k/2}  \,(4\pi)^{(k-1)/2}\,(k-1)!} \,.
\]
\end{prop}
\begin{proof} 
Let $\alpha \in (0,1/(3k))$ (we are really interested in the limit as $\alpha$ decreases to 0).  
Let $A=\alpha N$ and $B=2\alpha N$.   
%, and $\alpha_1=\alpha$.  
For each $\collec{n}\in \mathfrak{N}(N,\alpha)$,
%$k$-tuple $(n_1,\ldots,n_k)$ occurring in the summation
%of Equation (\ref{eq.W**0}), 
the inequality of Equation (\ref{eq.W**1}) becomes
\begin{equation}
   \label{eq.W**11}
   | \mathcal{W}(\alpha,A,B,\collec{n})|  \;\geq \;  t_N^{k-1}Z^*_k(1+o(1)) \,-\, \binom{k}{2}\left[(2N)^{k-2}(16\alpha N)+O(N^{k-2})\right]  
\end{equation}
with the help of Corollary \ref{cor.andre1}.  Recalling that $t_N=\lfloor (N/k)-\alpha N-B\rfloor$, we obtain 
from Equation (\ref{eq.W**0}) and Lemma \ref{lem.sequnifB} that 
\begin{align}
   \nonumber
     |\Dom (N,\alpha,A,B)|  \; \;\geq  \;  
 %     \\
  %   \nonumber
     &  \sum_{\collec{n}\in \mathfrak{N}(N,\alpha)} 
     %\sum_{\substack{n_1,\ldots,n_k\;: \\ \; n_1+\cdots+n_k=N, \\ \left|n_i-\frac{N}{k}\right| \leq \alpha N
  %\;\forall i  }} 
    \binom{N}{n_1,\cdots,n_k}^2 \left(1-\frac{C(\alpha)}{\alpha^2\sqrt{N}}\right)^2  \times 
   \\    
    & \hspace{22mm} N^{k-1}\left[  
     \left(\frac{1}{k}-3\alpha\right)^{k-1} Z_k^*-
       \binom{k}{2}2^{k+2}\alpha +o(1)\right]  \,.
       \label{eq.W**01}
 \end{align}
 Observe that the assumption of Corollary \ref{cor.Dombound} holds for large $N$ because $B=2A$.
 Therefore, applying Corollary \ref{cor.Dombound}, Equation  (\ref{eq.Hoeff2}), and Theorem \ref{thm.richmond} to Equation
 (\ref{eq.W**01}) gives 
\begin{eqnarray}
   \nonumber
 \liminf_{N\rightarrow\infty}  
    \frac{| {\E}_N(\deck)|}{N^{k-1} k^{2N+k/2}(4\pi N)^{(1-k)/2} } 
    & \geq & 
    \liminf_{N\rightarrow\infty}  
    \frac{| \Dom(N,\alpha,\alpha N, 2\alpha N)|
      }{k! \,N^{k-1} k^{2N+k/2}(4\pi N)^{(1-k)/2} } 
    \\
    \label{eq.SNDomsqueeze}
    &\geq &  \frac{1}{k!}
     \left[   \left(\frac{1}{k}-3\alpha\right)^{k-1} Z^*_k-\binom{k}{2}2^{k+2}\alpha\right]  \,.
 \end{eqnarray}
Since this inequality holds for arbitrarily small positive $\alpha$, 
the proposition follows.
\end{proof}

Finally, Theorem \ref{thm.E4321} follows immediately from
Propositions \ref{prop.4321upper} and \ref{prop.4321lower}.

\section{Weak convergence}
   \label{sec.weak}

In this section we present a convergence result for 
${\E}_N(\deck)$ in the spirit of permutons.  Section
\ref{sec.weakintro} describes the measure-theoretic
framework that we use, including an introduction to 
the Wasserstein distance, and presents the formal statement 
of the main theorem of this section along with the 
strategy of its proof.  Section \ref{sec.Wass}
presents some basic properties of Wasserstein distances 
that we shall need, particularly in the context of mixtures
(i.e.\ convex combinations) of probability measures.
Section \ref{sec.wassproof} proves the main result, 
following the strategy described in Section \ref{sec.weakintro}.

\subsection{Overview and statement of the main result}
    \label{sec.weakintro}

We start with some terminology and notation about measures.
We denote the set of all probability measures on a set $\chi$
by PM$(\chi)$.  (We should refer to the set of all probability 
measures on a measurable space, but the $\sigma$-algebra 
associated with $\chi$ will always be implicit and unambiguous.)
For $x\in \chi$, let $\delta_{x}$ be the measure on subsets of $\chi$
that assign value 1 to every set containing the point $x$ and value 0 to every other set.
We call $\delta_{x}$ the ``point mass at $x$.''

For a permutation $\sigma$ of size $N$, the ``empirical measure of $\sigma$'' is the 
measure $\mu_{\sigma}$ on $\R^2$ defined by
\begin{equation}
      \label{eq.musigma}
  \mu_{\sigma}   \; :=\;   \frac{1}{N}\sum_{i=1}^N  \delta_{(i,\sigma(i))}   \,.    
\end{equation}
Observe that $\mu_{\sigma}$ has total mass 1, i.e.\ it is a probability measure.  We can think of $\mu_{\sigma}$ as describing
the selection of one point of the plot of $\sigma$ uniformly at
random.
We also define the scaled empirical  measure $\hat{\mu}_{\sigma}$ by scaling $[0,N]^2$ down to 
the unit square:
\begin{equation}
     \label{eq.muhatsigma}
   \hat{\mu}_{\sigma}   \; :=\;   \frac{1}{N}\sum_{i=1}^N  \delta_{(i/N,\,\sigma(i)/N)}   \,.
\end{equation}
A permuton is a probability measure on $[0,1]^2$ whose marginal distributions are each
the uniform measure on $[0,1]$ (in the sense of 
equation (\ref{eq.marginal}) below).  
It is of interest to find (possibly random) 
permutons corresponding to weak limits of probability
measures on interesting classes of permutations
(see for example \cite{Glebov, Hoppen, KKRW, PS}).

We shall also use equations (\ref{eq.musigma}) and (\ref{eq.muhatsigma}) 
to define $\mu_{\sigma}$ and $\hat{\mu}_{\sigma}$ for 
bounded affine permutations $\sigma$ in ${\E}_N$.  
(We only use the point masses $(i,\sigma(i))$ for $i\in[N]$.)
We will be interested in weak limits of these
measures, but the limits cannot be permutons because they are not restricted to the unit square.  Rather, $\hat{\mu}_{\sigma}$
and the limits will be measures on the parallelogram 
\[  \pargram   \;=\; \{(x,y)\in\R^2 \,:\, 0\leq x\leq 1, \, |y-x|\leq 1 \}.  
\]
For $\sigma\in\E_N$, we have $\hat{\mu}_{\sigma}\in \hbox{PM}(\pargram)$.

\begin{rem}
Rather than having marginal distributions 
that are both uniform, our weak limits $\mu$
of scaled empirical measures of bounded
affine permutations will have the property that 
for every Borel subset $B$ of $[0,1]$, 
the values of $\mu(B\times \mathbb{R})$ and
$\mu\left([0,1]\times 
\cup_{j\in \mathbb{Z}}(B+j)\right)$ both equal 
the Lebesgue measure of $B$.  
This is because for every affine permutation
$\sigma$ of size $N$, the function on $[N]$ 
defined by $i\mapsto 1+(\sigma(i)\mod N)$ is an 
ordinary permutation.
\end{rem}

Our situation is more complicated than this, because we need to think
in terms of \textit{random measures}.  In particular, we are 
interested in what $\hat{\mu}_{\sigma}$ typically looks like 
for a randomly chosen $\sigma$ in $\E_N(\deck)$.  We formalize this 
by considering probability measures on the set of probability 
measures; that is, our random measures will be members of PM(PM($\pargram$)).  
%Accordingly, we extend the notation of the previously mentioned point mass $\delta_{x}$ as follows.
As $\delta_x$ is in $\hbox{PM}(\pargram)$ when 
$x\in \pargram$, so we have that $\delta_x$ is in 
$\hbox{PM}(\hbox{PM}(\pargram))$ when
$x\in \hbox{PM}(\pargram)$.   For example, 
the measure
$\nu=0.5\delta_{\hat{\mu}_{\oplus 132}}+
0.5\delta_{\hat{\mu}_{\oplus 321}}$
is the random measure that is equally likely to produce the
scaled empirical measure of either $\oplus 132$ or $\oplus 321$.
%\neal{Unfortunately we have not defined this [1423] notation yet.}
If $\mathcal{A}$ is a (nonempty finite) set of bounded affine permutations, we
shall write $\hat{\mu}[\mathcal{A}]$ to denote the scaled 
empirical measure $\hat{\mu}_{\sigma}$ where $\sigma$ is chosen 
uniformly at random from $\mathcal{A}$.  That is,
\[
     \hat{\mu}[\mathcal{A}]  \; :=
     \frac{1}{|\mathcal{A}|} \;\sum_{\sigma\in\mathcal{A}} \delta_{\hat{\mu}_{\sigma}} \,.
\]
Observe that $\hat{\mu}[\mathcal{A}]$ is in $\hbox{PM}(\hbox{PM}(\pargram))$, whereas 
$|\mathcal{A}|^{-1} \sum_{\sigma\in\mathcal{A}} \hat{\mu}_{\sigma}$ 
is in PM($\pargram$).

We shall use Wasserstein distance to show weak convergence.  Wasserstein distance is 
a metric on probability measures (on a given metric space) that corresponds to the topology of weak convergence 
(provided that the underlying metric 
space is bounded); see for example  Theorem 5.6 of \cite{Chen}.
It is defined as follows.  Let $(\chi,\rho)$ be a metric space. 
(In this paper, we shall need the example that $\chi$ is 
the parallelogram $\pargram$ and $\rho$ is Euclidean distance; and we shall also need the example that 
$\chi$ is PM($\pargram$), with $\rho$ being the 
Wasserstein distance on PM($\pargram$).)  
Let $\nu_1$ and $\nu_2$ be two probability
measures on $\chi$.  Let Joint$(\nu_1,\nu_2)$ be the set of all probability measures ${\cal J}$
on $\chi\times \chi$ whose marginal distributions are $\nu_1$ and $\nu_2$, i.e.
\begin{equation}
  \label{eq.marginal}
 \nu_1(B)  \,=\,  \int_{B\times \chi} {\cal J}(dx,dy)     \hspace{5mm}\hbox{and}\hspace{5mm}
      \nu_2(B)  \,=\,  \int_{\chi \times B} {\cal J}(dx,dy) 
      \hspace{5mm}\hbox{for every Borel set $B\subset \chi$}\,.
\end{equation}
Then the Wasserstein distance between $\nu_1$ and $\nu_2$ is defined to be
\begin{equation} 
   \label{eq.wassinf}
      \wass(\nu_1,\nu_2)    \;=\;    \inf\left\{   \int_{\chi}\int_{\chi} \rho(x,y)\,{\cal J}(dx,dy) \;:\;{\cal J}\in 
       \hbox{Joint}(\nu_1,\nu_2)\right\}  \,.
\end{equation}
That is, $\wass(\nu_1,\nu_2)$ is the infimum of $E(\rho(X_1,X_2))$ over all jointly distributed
pairs of random variables
$(X_1,X_2)$ on $\chi\times\chi$
where $X_1$ and $X_2$ have distributions $\nu_1$ and $\nu_2$
respectively.  
It is known that this infimum is always attained by some joint 
distribution ${\cal J}$  (e.g.\ Lemma 5.2 of \cite{Chen}).
We shall also use the following convention:  if $Y$ and $Z$
are two $\chi$-valued random variables with respective probability 
distributions $\nu_Y$ and $\nu_Z$, then we may write 
$\wass(Y,Z)$ to denote $\wass(\nu_Y,\nu_Z)$.

We shall use the following abbreviating notation.
We shall write PM$_{1}$ for PM($\pargram$), and 
PM$_2$ for PM(PM($\pargram$)).  Correspondingly,
we shall write Wass$_1$ to denote the 
Wasserstein distance on PM$_1$ (determined by the Euclidean metric on $\pargram$), and Wass$_2$ for the Wasserstein distance on PM$_2$
(determined by the $\wass_1$ metric on PM($\pargram$)).

\begin{rem}
   \label{rem.wass}
In general, defining weak convergence in 
PM($X$) requires specifying a topology
on $X$, or equivalently specifying which functions on $X$ are
continuous.  When $X$ is PM($Y$) for some set $Y$, we need to 
specify the topology of convergence of measures on $Y$.  
Weak convergence, corresponding to Wasserstein metric on $X=\hbox{PM}(Y)$, is a standard choice, and this is our choice. 
But there are other possibilities, such as Total Variation.
The choice of topology on PM($Y$) is a separate decision 
from the choice of topology on $Y$.
\end{rem}

\begin{defi} 
    \label{def.uniform}
Let $\textup{Unif}({\cal H})$ 
denote the uniform distribution on a set ${\cal H}$, and
let $P_{\textup{Unif}({\cal H})}$ denote the 
corresponding probability measure.
The set ${\cal H}$ will always be bounded, and it will be 
of one of two kinds:  a discrete set (i.e.\ a finite set), or a continuous set of dimension $r$ (i.e.\ a
Borel subset of an $r$-dimensional affine subset of 
$\R^d$ ($0<r\leq d)$
that has non-zero $r$-dimensional Lebesgue measure).  
If ${\cal H}$ is continuous, then the ``uniform distribution'' on ${\cal H}$ 
refers to the normalized restriction of $r$-dimensional 
Lebesgue measure to ${\cal H}$.
\end{defi}

Now we shall define the random measure $\lambda^{Q_0}$,
which, as we shall see,
is the weak limit of $\hat{\mu}[\E_N(\deck)]$ as $N\rightarrow\infty$.

\begin{defi} 
   \label{def.lammix}
(a) Given $z\in[-1,1]$, let $\lambda_z \in \textup{PM}_1$ be the probability 
measure on $\pargram$ that is uniformly distributed on the line
% segment $\{(w,w+z):w\in [0,1]\}$ in $\R^2$.
segment from $(0,z)$ to $(1,1+z)$.
(Observe that the union of all such line segments is $\pargram$.)
\\
(b) Given $\collec{z}=(z_1,\ldots,z_k)\in [-1,1]^k$, define $\lambda\langle z \rangle \in \textup{PM}_1$ by
\begin{equation}
    \label{eq.lamcollz}
    \lambda\collec{z} \;:=\; \frac{1}{k}\sum_{i=1}^k \lambda_{z_i}\,,
\end{equation}
the probability measure uniformly distributed on the $k$ line 
segments
in $\pargram$ of slope 1 with $y$-intercepts $z_1,\ldots,z_k$.
\\
(c) Given a set $Q\subset [-1,1]^k$  (discrete or continuous),
%$\{\vec{x}\in [-1,1]^k:\sum_{i=1}^kx_i=0\}$, 
let $\lambda^{Q} \in \textup{PM}_2$ be the random measure given by
$\lambda\collec{\beta}$ where $(\beta_1,\ldots,\beta_k)$ is 
uniformly distributed on $Q$.  That is, in particular,
\begin{equation}  
   \nonumber    %\label{eq.lambdaQ}
    \lambda^{Q}  \;=\;  \frac{1}{|Q|} \sum_{\collec{z}\in Q}
    \delta_{\lambda\collec{z}}
    \hspace{5mm}\hbox{if $Q$ is discrete.}
\end{equation}
(d) Define $Q_0:=\left\{\collec{x}\in [-1,1]^k:\sum_{i=1}^kx_i=0\right\}$.
\end{defi}

The main result of this section, 
Theorem \ref{thm.wassmain} below, states that 
$\hat{\mu}[\E_N(\deck)]$
converges weakly to $\lambda^{Q_0}$. Intuitively, 
% $\lambda^{Q_0}$ looks like 
% $k$ lines of slope $1$ with $y$-intercepts uniformly 
%chosen from $[-1,1]$ subject to the constraint that their sum is $0$, and 
this result means that 
the plot of 
a random element of $\E_n(\deck)$ 
(scaled down to the unit square) 
looks like 
the support of
$\lambda^{Q_0}$, 
which consists of $k$ lines of slope $1$ with $y$-intercepts 
%uniformly 
chosen randomly
from $[-1,1]$ subject to the constraint that their sum is $0$.
We shall prove this using Wasserstein distances. 

\begin{thm}  
   \label{thm.wassmain}
Wass$_2(\hat{\mu}[\E_N(\deck)],\,\lambda^{Q_0})$ converges
to 0 as $N\rightarrow\infty$.   That is, the sequence of random measures $\hat{\mu}[\E_N(\deck)]$ converges weakly to 
$\lambda^{Q_0}$, with respect to the topology of weak convergence 
on PM$(\pargram)$.
\end{thm}

\begin{rem}
   \label{rem.thmwassmain}
With reference to Remark \ref{rem.wass}, the weak convergence of
$\hat{\mu}[\E_N(\deck)]$ to $\lambda^{Q_0}$ does not hold 
with respect to the Total Variation topology on PM$(\pargram)$.  
This is because the set of points in $\pargram$ with rational 
coordinates has probability 1 under every $\hat{\mu}_{\sigma}$
but has probability 0 under every $\lambda_{z}$, and hence the 
total variation distance between $\hat{\mu}_{\sigma}$ and 
%$\frac{1}{k}\sum_{i=1}^k\lambda_{\beta_i}$ 
$\lambda\collec{z}$  is always 1.
\end{rem}

Here, in brief, are the main parts of the strategy of the proof.
\\
\underline{Step 1}.
First, recall the set
$\Dom\equiv\Dom(N,\alpha,A,B)$ from 
Definition \ref{def.Dom},
and the $k!$-to-one map $\Psi:\Dom\rightarrow\E_N(\deck)$ 
whose image is most of $\E_N(\deck)$.  
We shall show that the random measures 
$\hat{\mu}[\E_N(\deck)]$ and $\hat{\mu}[\Psi(\Dom)]$
are close in Wass$_2$ distance.
\\
\underline{Step 2}.  Fix $\vec{v}=(\collec{n},\collec{\Delta},\collec{G},\collec{H})\in \Dom$ and let $\sigma=\Psi(\vec{v})$ be 
its associated affine permutation.  
We shall show that
$\hat{\mu}_{\sigma}$ is close to 
$\frac{1}{k}\sum_{i=1}^k\lambda_{(\Delta_i/n_i)}$ 
in Wass$_1$ distance.
\\
\underline{Step 3}.  
Let $\collec{n}\in \mathfrak{N}(N,\alpha)$.
%Fix $\collec{n}$ satisfying our standard conditions, namely
%$n_1+\cdots+n_k=N$ and $\left|n_i-\frac{N}{k}\right|\leq \alpha N$.
Let $\mathcal{W}\equiv\mathcal{W}(\alpha,A,B,\collec{n})$ be 
as defined in the statement of Lemma \ref{lem.W**},  namely 
the set of $\collec{\Delta}$ such that $(\collec{n},\collec{\Delta},
\collec{G},\collec{H})$ is in $\Dom$ for some $\collec{G}$
and $\collec{H}$.  Then we 
%define a continuous set $\mathcal{Y}\subset Q_0$
show that the (continuous) set $Q_0$ is well approximated by 
the discrete set 
% scaled down version of $\mathcal{W}$ defined by
\begin{equation*}
 \widehat{\mathcal{W}} \;:=  \;   
 \left\{\left(\frac{\Delta_1}{n_1},\ldots,\frac{\Delta_k}{n_k}\right):
    \collec{\Delta}\in\mathcal{W}\right\},
\end{equation*}
which is a scaled down version of $\mathcal{W}$.
We shall show that for every $\collec{n}\in\mathfrak{N}(N,\alpha)$,
$\lambda^{\widehat{\mathcal{W}}}$  is close to $\lambda^{Q_0}$ 
(and close to $\hat{\mu}[\Psi(\Dom)]$, by Step 2) 
in Wass$_2$ distance.
%\\
%\underline{Step 4}.  Finally, we show that the set 
% $Q_0\setminus\mathcal{Y}$ is small, and hence that 
% $\lambda^{\mathcal{Y}}$ is close to $\lambda^{Q_0}$ in Wass$_2$ distance.

\subsection{Mixtures and Wasserstein Distance}
   \label{sec.Wass}

Throughout this subsection, we assume that 
$\chi$ is a set with metric $\rho$.  We define the diameter 
of $\chi$ to be
\[ \textup{diam}(\chi)  \;:=  \;  \sup\{  \rho(a,b) : a,b\in \chi\}.
\]
In the rest of this paper, we assume the diameter of $\chi$ is finite.
We shall use the Borel sigma-algebra of $\chi$, with open sets 
determined by the metric $\rho$.

Given $m\in \mathbb{N}$,
let $\nu_1,\ldots,\nu_m$ be probability measures on $\chi$,
and let $a_1,\ldots,a_m$ be real numbers in $[0,1]$ whose sum is 1.
Let $\nu=\sum_{i=1}^m a_i\nu_i$.  Then $\nu$ is also a probability 
measure.  In other words, every convex combination of probability 
measures (on a given measurable space) is a probability measure.

In statistical terminology, the measure $\nu$ in the preceding
paragraph is also called a ``mixture'' of $\nu_1,\ldots,\nu_m$.
It may be interpreted with the following construction.
\begin{verse}
   \underline{Randomized Algorithm MIX}
   \\
   (a)  Let $\vec{X}=(X_1,\ldots,X_m)$ be a random vector
   (taking values in $\chi^d$)  such that
   the component $X_i$ has distribution $\nu_i$ for each $i\in[m]$.  
   (We do not require that the components be independent.)
   \\
   (b)  Let $J$ be a random variable, independent of $\vec{X}$, 
   such that $\Pr(J=i)=a_i$ for each $i\in[m]$.
   \\
   (c)  Let $Y=X_J$.  (That is, we assign $Y$ to be $X_1$ with probability $a_1$, to be $X_2$ with probability $a_2$, and 
   so on.)  Then the distribution of $Y$ is $\nu$.
\end{verse}
To prove the conclusion of (c), let $D$ be a measurable subset of $\chi$.  Then
\begin{eqnarray*}
\Pr(X_J\in D)  \;=\;  \sum_{i=1}^m\Pr(J=i \hbox{ and }X_i\in D)
    & = & \sum_{i=1}^m \Pr(J=i)\Pr(X_i\in D) \\
    & = &   
    \sum_{i=1}^m a_i\nu_i(D) \;=\;  \nu(D) \,.
\end{eqnarray*}
% We can describe $\nu$ as being the distribution of the 
% random measure $\nu_J$.  
The above construction leads directly to the following lemmas.
   
\begin{lemma}
   \label{lem.wass1}
Let $\nu_1,\ldots,\nu_m,\omega_1,\ldots,
\omega_m\in\textup{PM}(\chi)$.  
Let $a_1,\ldots,a_m$ be nonnegative real numbers that add up to 1.
Then 
\[  \wass\left(\sum_{i=1}^ma_i\nu_i,
   \sum_{i=1}^ma_i\omega_i\right)  \;\leq \; 
   \sum_{i=1}^ma_i\wass(\nu_i,\omega_i)\,.
\]
\end{lemma} 

\begin{proof}
Let $\nu=\sum_{i=1}^ma_i\nu_i$ and $\omega=\sum_{i=1}^ma_i\omega_i$.
For each $i\in [m]$, let $(X_i,Z_i)$ be a $\chi{\times}\chi$-valued 
random vector such that $E(\rho(X_i,Z_i))=\wass(\nu_i,\omega_i)$ 
(we know that such a random vector exists because the infimum in 
Equation (\ref{eq.wassinf}) is always attained).
Also let
the $m$ random vectors $(X_i,Z_i)$ ($i=1,\ldots,m)$ be independent.
Lastly, 
let $J$ be a random variable, independent of the $(X_i,Z_i)$'s,
such that $\Pr(J=i)=a_i$ for each $i$.
Then the $\chi{\times}\chi$-valued random vector $(X_J,Z_J)$ has
marginal distributions $\nu$ and $\omega$.
Therefore 
\begin{align*}
   \textup{Wass}(\nu,\omega) & \leq E\left(\rho(X_J,Z_J)\right) \\
   & = \sum_{i=1}^m \Pr(J=i)E(\rho(X_i,Z_i)) \\
   & = \sum_{i=1}^m a_i \wass(\nu_i,\omega_i)\,. \qedhere
\end{align*}
\end{proof}

\begin{lemma}
   \label{lem.wass2}
Let $\nu_1,\ldots,\nu_m\in\textup{PM}(\chi)$.  
Let $a_1,\ldots,a_m,b_1,\ldots,b_m$ be nonnegative real numbers 
such that $\sum_{i=1}^ma_i=1=\sum_{i=1}^mb_i$.
Then 
\[  \wass\left(\sum_{i=1}^ma_i\nu_i,
   \sum_{i=1}^mb_i\nu_i\right)  \;\leq \; 
  \hbox{diam}(\chi) \sum_{i=1}^m|a_i-b_i| \,.
\]
\end{lemma} 

\begin{proof}
We first assert that there exists a random vector $(J,K)$ such that 
$\Pr(J=i)=a_i$ and $\Pr(K=i)=b_i$ for each $i$ and $\Pr(J\neq K)=
\sum_{i=1}^m|a_i-b_i|$.  This true by Propositions 4.2 and 4.7 and
Remark 4.8 of \cite{LPW}.

Next, let $(X_1,\ldots,X_m)$ be a random vector, independent of
$(J,K)$, such that
the component $X_i$ has distribution $\nu_i$ for each $i$.  
Then we have
\begin{align*}
  \wass\left(\sum_{i=1}^ma_i\nu_i,
   \sum_{i=1}^mb_i\nu_i\right) &\leq  E(\rho(X_J,X_K)) \\
     &\leq  
     \textup{diam}(\chi)\;\Pr\left(\rho(X_J,X_K)\neq 0\right)  \\
     & \leq \textup{diam}(\chi) \,\Pr(J\neq K) \,. \qedhere
\end{align*}
\end{proof}

The following lemma shows that the Wasserstein distance between
two uniform distributions is small when their support sets have 
large overlap.  
Recall the terminology of Definition \ref{def.uniform}.  

\begin{lemma}
   \label{lem.unif}    
Let $A\subset B\subset \chi$. 
Assume that $A$ and $B$ are either both nonempty finite sets,
or both continuous sets in the sense of Definition
\ref{def.uniform}.  Then
%Writing $P_{\textup{Unif}(B)}$ for the 
%probability measure of the  Unif$(B)$ distribution, 
we have
\[    \wass(\textup{Unif}(A),\textup{Unif}(B))  \;\leq \;
     \textup{diam}(B)\,P_{\textup{Unif}(B)}(B\setminus A) \,.
\]
\end{lemma}

\begin{proof}
Let $V_A$ and $V_B$ be independent random variables,
with the Unif$(A)$ and Unif$(B)$ distributions respectively.
Define the random variables $X$ and $Y$ by 
\[
     X\;:=\;V_B  \hspace{5mm}\hbox{and}\hspace{5mm}
         Y \;:=\; \begin{cases}
                V_B   & \hbox{if } V_B\in A \,, \\
                V_A  & \hbox{if } V_B\not\in A \,.
                \end{cases}
\]
It is routine to check that $Y$ has the Unif$(A)$ distribution.
Since $X$ has the Unif$(B)$ distribution, we have 
$\wass(\textup{Unif}(A),\textup{Unif}(B)) \leq E(\rho(Y,X))$.
The result now follows from the fact that 
$\rho(Y,X)\leq \textup{diam}(\chi) \,\mathbb{I}_{(V_B\not\in A)}$ , 
where $\mathbb{I}_{\cal E}$ is the indicator random variable that
equals 1 or 0 according to whether the event 
${\cal E}$ occurs or not.
\end{proof}

We remark that this lemma is not useful when $A$ is 
discrete and $B$ is continuous, nor when $A$ and $B$ are 
continuous sets with different dimensions, since in such 
cases $P_{\textup{Unif}(B)}(B\setminus A)=1$.

\subsection{Proofs of Wasserstein Approximations}
   \label{sec.wassproof}
   
In this section, we fix $k\geq 2$.

Recall the set $\Dom(N,\alpha,A,B)$ from 
Definition \ref{def.Dom}, as well as the function 
$\Psi$ defined on this set by Equations (\ref{eq.hperiodic}) and (\ref{eq.sigma4def}), for suitable values of
$N$, $\alpha$, $A$, and $B$.
We shall first handle Step 2 in the proof strategy outlined at 
the end of Section \ref{sec.weakintro}.

\begin{prop}
   \label{prop.domwass1}
Let $N\in \mathbb{N}$, let $A$ and $B$ be
positive real numbers, and let 
$\alpha$ be a real number in $(0,1/k)$.
Let $\vec{v}=(\collec{n},\collec{G},\collec{H},\collec{\Delta}) 
\in \Dom(N,\alpha,A,B)$, and let $\sigma=\Psi(\vec{v})$.
As in Equation (\ref{eq.lamcollz}), 
let $\lambda{\collec{\Delta/n}}$ be the probability measure on 
$\pargram$ defined by
\[  \lambda{\collec{\Delta/n}} \;:=\;  \frac{1}{k}\sum_{i=1}^k
       \lambda_{\Delta_i/n_i}.
\]
Then 
\begin{equation}
    \label{eq.wass1bd}
    \wass_1\left( \hat{\mu}_{\sigma} , \lambda{\collec{\Delta/n}} \right)   \;\leq \;  
    \frac{1}{N} \left(2A+\frac{4 k}{1-k\alpha}     \right)
      \;+\; 4k\alpha \,.
\end{equation}
\end{prop}

\begin{proof}
According to the definition of the function $\Psi$ preceding Lemma
\ref{lem.sigmaaffine}, we can write
\[  \hat{\mu}_{\sigma} \;=\; \frac{1}{N}\sum_{i=1}^k \sum_{j=1}^{n_i}
    \delta_{\left(g_{ij}/N\,,\,h_{i,(j+\Delta_i)}/N \right)}\,.
\]
Define the probability measure
\[    \Lambda \;=\;  \frac{1}{N}\sum_{i=1}^k \sum_{j=1}^{n_i}
    \delta_{\left(j/n_i\,,\,(j+\Delta_i)/n_i \right)}\,. 
\]
By Lemma \ref{lem.wass1} and the general property that 
$\wass(\delta_x,\delta_v)=\rho(x,v)$, we have
\begin{eqnarray}
   \nonumber
   \wass_1\left(\hat{\mu}_{\sigma},\Lambda\right)  & \leq & 
     \frac{1}{N}\sum_{i=1}^k \sum_{j=1}^{n_i} \left(
    \left|\frac{g_{ij}}{N}-\frac{j}{n_i} \right| \,+\,
    \left| \frac{h_{i,(j+\Delta_i)}}{N}-\frac{j+\Delta_i}{n_i} \right|
     \right)
     \\
     \nonumber
       & \leq & 
     \frac{1}{N}\sum_{i=1}^k \sum_{j=1}^{n_i} \frac{2}{N}
       \left( A+\frac{k}{1-k\alpha}\right)   \hspace{7mm}
       \hbox{(by Lemma \ref{lem.V**})}
       \\
      \label{eq.muLam} 
       & = & \frac{2}{N}  \left( A+\frac{k}{1-k\alpha}\right) \,.
\end{eqnarray}

Let $U$ be a random variable with uniform distribution on the
interval $(0,1)$.
For $i\in [k]$, define the two $\R^2$-valued random vectors 
\[    \vec{\gamma}\;:=\; \left( U, U+\frac{\Delta_i}{n_i} \right)
    \hspace{5mm}\hbox{and}\hspace{5mm}
    \vec{\kappa} \;:=\;  \left(  \frac{\lceil n_iU\rceil}{n_i}\,,\,
   \frac{\lceil n_iU\rceil}{n_i} +\frac{\Delta_i}{n_i} \right)
\]
(here, $\lceil \cdot \rceil$ is the ceiling function).
Then $\vec{\gamma}$ has distribution $\lambda_{\Delta_i/n_i}$ and
$\vec{\kappa}$ has distribution $\Lambda_i$, where
\[    \Lambda_i \:=\; \frac{1}{n_i}\sum_{j=1}^{n_i}    
    \delta_{\left(j/n_i\,,\,(j+\Delta_i)/n_i \right)}\,.
\]
Since 
 $\left|u-\frac{\lceil n_i u\rceil}{n_i}\right| \leq \frac{1}{n_i}$
for every real $u$, we see that 
$\rho(\vec{\gamma},\vec{\kappa})\leq 2/n_i$ with probability $1$, 
and hence $\wass_1(\lambda_{\Delta_i/n_i},\Lambda_i)\leq 2/n_i$. 
Noting that $\Lambda=\sum_{i=1}^k \frac{n_i}{N}\,\Lambda_i$, 
we deduce from Lemma \ref{lem.wass1} that 
\begin{equation}
     \label{eq.wasslam}
     \wass_1\left(\sum_{i=1}^k \frac{n_i}{N}\lambda_{\Delta_i/n_i}\,,\,\Lambda\right) \;\leq\;
  \sum_{i=1}^k \frac{n_i}{N}\,\frac{2}{n_i} \;=\; \frac{2k}{N} \,.
\end{equation}

Next, since $\textup{diam}(\pargram)=\sqrt{10}$, 
Lemma \ref{lem.wass2} tells us that
\begin{equation}
    \label{eq.wasslam2}
    \wass_1\left(\lambda\collec{\Delta/n}\,,\,
    \sum_{i=1}^k \frac{n_i}{N}\lambda_{\Delta_i/n_i} \right)
    \;\leq \;
    \sqrt{10}\,\sum_{i=1}^k \left| \frac{1}{k} -\frac{n_i}{N} \right|
    \;< \;  4k\alpha
\end{equation}
(where the second inequality uses the definition of
\Dom, specifically Equation \eqref{eqn:defNNa}).  Finally, the proposition follows from the triangle 
inequality and Equations (\ref{eq.muLam}--\ref{eq.wasslam2}).
\end{proof}

To prepare us for Step 3, we first prove a lemma about 
random measures of the form $\lambda^Q$ as defined in 
Definition \ref{def.lammix}(c).
%Equation (\ref{eq.lambdaQ}).

\begin{lemma}
    \label{lem.lambdaQ}
Let $QA$ and $QB$ be two (discrete or continuous)
subsets of $[-1,1]^k$.  Then 
\[   \wass_2\left(\lambda^{QA},\lambda^{QB}\right)  \;\leq\;
     \wass_1\left(\textup{Unif}(QA),\textup{Unif}(QB)\right).
\]
\end{lemma}

\begin{proof}
Let $U$ be a uniformly distributed random variable on $[0,1]$,
and let $J$ be a uniformly distributed random variable on 
$\{1,\ldots,k\}$, independent of $U$.  
For $\collec{x}=(x_1,\ldots,x_k)\in[-1,1]^k$, the random 
point  $(U,U+x_J)$ has distribution $\lambda\collec{x}$.  
If also $\collec{v}\in [-1,1]^k$, then
\begin{eqnarray}
    \nonumber 
    \wass_1(\lambda\collec{x},\lambda\collec{v})  & \leq & 
      E|| (U,U+x_J)-(U,U+v_J)||   \\
      \nonumber 
      & = &  E|x_J-v_J|   \\
      \nonumber
      & = & \sum_{i=1}^k \frac{1}{k} |x_i-v_i| \\
      \label{eq.lamdif}
      & \leq & ||\collec{x}-\collec{v}||\,.
\end{eqnarray}

Let $(\collec{\beta^A},\collec{\beta^B})$ be an $(\R^k\times\R^k)$-valued
random vector 
such that $\collec{\beta^A}$ is uniformly distributed on $QA$,
$\collec{\beta^B}$ is uniformly distributed on $QB$,
%whose distribution is in 
%Joint$(\textup{Unif}(QA),\textup{Unif}(QB))$ 
and 
\[ E||\collec{\beta^A}-\collec{\beta^B}|| \,=\,
\wass_1\left(\textup{Unif}(QA),\textup{Unif}(QB)\right).
\]
%where $\rho$ is Euclidean norm in $\R^k$.
%Let $U$ be a uniformly distributed random variable on $[0,1]$,
%and let $J$ be a uniformly distributed random variable on 
%$\{1,\ldots,k\}$, such that $U$, $J$, and  $(\vec{\beta}^A,\vec{\beta}^B)$ are independent.
%Then $(U,U+\beta^A_J)$ 
Since the random measures $\lambda\collec{\beta^A}$ 
and $\lambda\collec{\beta^B}$ have distributions $\lambda^{QA}$ and
$\lambda^{QB}$ respectively, we obtain
\begin{align*}
  \wass_2\left(\lambda^{QA},\lambda^{QB}\right) & \leq 
  E\left(  \wass_1\left( \lambda\collec{\beta^A}\,,\,
    \lambda\collec{\beta^A} \right)\,\right)  \\ 
%       E||(U,U+\beta^A_J)-(U,U+\beta^B_J)||   \\
%       & = & E|| (0,\beta^A_J-\beta^B_J)|| \\
%       & = & \frac{1}{k}\sum_{i=1}^k E|\beta_i^A-\beta_i^B|   \\
       & \leq      %  \frac{1}{k}\sum_{i=1}^k
             E||\collec{\beta^A}-\collec{\beta^B}||  
             \hspace{25mm}\hbox{(by Equation (\ref{eq.lamdif}))} \\
     & = \wass_1\left(\textup{Unif}(QA),\textup{Unif}(QB)\right)\,. \qedhere
\end{align*}
\end{proof}
%Let $\mathfrak{N}(N,\alpha)$ be the set of $\collec{n}=(n_1,\ldots,n_k)\in \mathbb{N}^k$ such that 
%$n_1+\cdots+n_k=N$ and $\left|n_i-\frac{N}{k}\right|\leq \alpha N$
%for every $i\in[k]$. 

Recall the definitions of $\mathfrak{N}(N,\alpha)$ and
$\mathcal{W}(\alpha,A,B,\collec{n})$ from Definition
\ref{def.lower}(c) and Lemma \ref{lem.W**} respectively.

\medskip

\begin{prop}
  \label{prop.wassnn}
Fix $k\geq 2$.  There is a  positive constant $C$, depending only
on $k$, such that the following holds.
Let $N$ be a natural number, let $A$ and $B$ be positive real numbers, and let $\alpha\in(0,1/4k)$.
Let $\collec{n}\in \mathfrak{N}(N,\alpha)$ and write
$\mathcal{W}$ for $\mathcal{W}(\alpha,A,B,\collec{n})$.
Let
\begin{equation}
  \label{eq.lambdawide}
  \lambda^{\widehat{\mathcal{W}}} \;=\;  \frac{1}{|\mathcal{W}|}
     \sum_{\collec{\Delta}\in\mathcal{W}}
     \delta_{\lambda\collec{\Delta/n}} \,.
\end{equation}
Then  
\begin{equation} 
  \label{eq.wasslamwidebound}
  \wass_2\left(\lambda^{\widehat{\mathcal{W}}}\,,\,
  \lambda^{Q_0}\right)  \;\leq \;  C\left( \alpha +\frac{A+B+1}{N}\right) \,.
\end{equation}
\end{prop}

\begin{proof}
We begin by setting some notation.
For $\collec{x}=(x_1,\ldots,x_k)\in \mathbb{Z}^k$, 
let $\collec{\widehat{x}}$ and $\collec{x^*}$
be the rescaled vectors
\[     \collec{\widehat{x}}   \;=\; \left( \frac{x_1}{n_1},
    \cdots,\frac{x_k}{n_k}\right) \hspace{5mm}\hbox{and}\hspace{5mm}
   \collec{x^*}   \;=\; \left( \frac{k\,x_1}{N},
    \cdots,\frac{k\,x_k}{N}\right)  \,,
\]
and let the corresponding sets of rescaled $\collec{\Delta}$ vectors be
\[   \widehat{\mathcal{W}}  \;=\;  \left\{ \collec{\widehat{\Delta}}
    \,:\, \collec{\Delta}\in \mathcal{W}  \right\}
    \hspace{5mm}\hbox{and} \hspace{5mm} 
   \mathcal{W}^*  \;=\;  \left\{ \collec{\Delta^*}
    \,:\, \collec{\Delta}\in \mathcal{W}  \right\} \,.
\]
With this notation, the definition of $\lambda^{\widehat{W}}$
in Equation (\ref{eq.lambdawide}) is consistent with 
the definition given in Definition \ref{def.lammix}(c).
%Equation (\ref{eq.lambdaQ}).

If $B \geq N/2k$, then the bound (\ref{eq.wasslamwidebound}) 
holds whenever $C\geq 2k \,\textup{diam}(\pargram)$.  Thus, without
loss of generality, we can and shall assume $B<N/2k$ in this proof.
Similarly, we shall assume that $N>4k^2$.

By Lemma \ref{lem.lambdaQ}, it suffices to prove the desired
upper bound for 
$\wass_1\!\left(\textup{Unif}(\widehat{\mathcal{W}}),\,
\textup{Unif}(Q_0)\right)$.
To do this, we shall define an intermediate continuous set 
$\mathcal{Y}$ of dimension $k-1$, and show that 
$\wass_1(\textup{Unif}(\widehat{\mathcal{W}}),
\textup{Unif}(\mathcal{W}^*))$, 
$\wass_1(\textup{Unif}(\mathcal{W}^*), \,
\textup{Unif}(\mathcal{Y}))$,
and $\wass_1(\textup{Unif}(\mathcal{Y}), \textup{Unif}(Q_0))$
are all small.  The third term will be handled with 
Lemma \ref{lem.unif}, while the other two will be treated directly.  

First we show that Unif$(\widehat{\mathcal{W}})$ is close to 
Unif$(\mathcal{W}^*)$.  For each $\collec{\Delta}\in \mathcal{W}$,
we have
\begin{eqnarray*}
   \wass_1\left(\delta_{\collec{\Delta^*}}\,,\,
        \delta_{\collec{\widehat{\Delta} } } \right)  \;=\;
    ||\collec{\Delta^*}-\collec{\widehat{\Delta}}||  & \leq & 
      \sum_{i=1}^k \left|  \frac{k\Delta_i}{N}-\frac{\Delta_i}{n_i}
       \right|    \\
       & = &   \sum_{i=1}^k  \frac{|\Delta_i|}{n_i}\,\frac{k}{N} \,
         \left|  n_i-\frac{N}{k}\right|   \\
        & <  &   \sum_{i=1}^k 1\cdot \frac{k}{N} \cdot\alpha N \\
        & = & k^2 \alpha \,.
\end{eqnarray*}
Using this bound together with Lemma \ref{lem.wass1} shows that
\begin{equation}
    \label{eq.wassbdWW}
     \wass_1\left(\textup{Unif}(\mathcal{W}^*)\,,\,
        \textup{Unif}(\widehat{\mathcal{W}})\right)  \;\leq \;
        k^2 \alpha \,.
\end{equation}

Next we define a continuous set $\mathcal{Y}$ of dimension $k-1$
that approximates the discrete set $\mathcal{W}^*$.
Let $\mathcal{P}_0$ be the hyperplane
\[    \mathcal{P}_0  \; :=\;  \{  (x_1,\ldots,x_k)\in \R^k \,:\,
     x_1+\cdots+x_k=0 \,\} \,.
\]
For each $\collec{z}\in \mathcal{P}_0$, let Cube$\collec{z}$ be the 
intersection of $\mathcal{P}_0$ with translation by $\collec{z}$
of the ``hypercubical tube'' $[0,k/N)^{k-1}\times \R$, i.e.
\begin{equation} 
  \label{eq.Cubedef}
  \textup{Cube}\collec{z} \;:=\;  \left\{ (x_1,\ldots,x_k)\in \mathcal{P}_0\,: \, z_i\leq x_i < z_i+\frac{k}{N}, \,i=1,\ldots,k-1 
    \right\} \,.
\end{equation}
Notice that for $x\in \textup{Cube}\collec{z}$, the relations
$x_k=-\sum_{i=1}^{k-1}x_i$ and $z_k=-\sum_{i=1}^{k-1}z_i$ imply that
\begin{equation}
    \label{eq.kcoord}
        z_k-\frac{k(k-1)}{N}  \;<\;  x_k  \;\leq \; z_k \,.
\end{equation}
It is important to observe that the collection of sets $\{\textup{Cube}\collec{z^*}\,:\,
\collec{z}\in \mathcal{P}_0 \cap \mathbb{Z}^k\}$ is a partition of
$\mathcal{P}_0$.

Let $\collec{\Delta}\in \mathcal{W}$ and let $\collec{x}\in\textup{Cube}\collec{\Delta^*}$.
Then $||\collec{x}-\collec{\Delta^*}||<2k(k-1)/N$ by Equations
(\ref{eq.Cubedef}) and (\ref{eq.kcoord}).  Therefore
\begin{equation} 
  \label{eq.wassCubedelta}
 \wass_1\left(\textup{Unif}({\textup{Cube}\collec{\Delta^*}}),
  \, \delta_{\collec{\Delta^*}} \right) \;\leq \; \frac{2k(k-1)}{N}\,.
\end{equation}
We now define the subset $\mathcal{Y}$ of $\mathcal{P}_0$ 
to be the union of Cube$\collec{z}$ over all 
$z\in \mathcal{W}^*$, i.e.
\[  \mathcal{Y} \;:=\;  \bigcup_{\Delta\in \mathcal{W}}
    \textup{Cube}\collec{\Delta^*} \,.  
\]
Since the sets Cube$\collec{\Delta^*}$ are all translates of one
another, we see that the uniform distribution on $\mathcal{Y}$
is the uniform mixture of the uniform distributions on its
constituent Cube sets:
\begin{equation}  
  \label{eq.unifYmix}
   P_{\textup{Unif}(\mathcal{Y})}   \;=\;
    \frac{1}{|\mathcal{W}|}\sum_{\collec{\Delta}\in \mathcal{W}}
    P_{\textup{Unif}(\textup{Cube}\collec{\Delta^*})}
\end{equation}
By Lemma \ref{lem.wass1} and Equations (\ref{eq.wassCubedelta}) and
(\ref{eq.unifYmix}), we see that
\begin{equation}
    \label{eq.wassUYUW}
     \wass_1\left(\textup{Unif}(\mathcal{Y}),\,
    \textup{Unif}(\mathcal{W}^*)\right) \;\leq \;\frac{2k(k-1)}{N} \,.
\end{equation}

Now we need to show that $\mathcal{Y}$ is a good approximation of 
$Q_0$.  First we claim
\begin{equation}
    \label{eq.YinQ0}
     \mathcal{Y} \;\subset \; Q_1, \hspace{5mm}\hbox{where we define}
     \hspace{3mm}Q_1\;:=\; \left(1+k\alpha+\frac{k^2}{N}\right) Q_0\,,
\end{equation}
using the standard notation for homothety:  
for positive $t$, $tQ_0=\{t\collec{x}:\collec{x}\in Q_0\}$.
Let $\collec{x}\in\mathcal{Y}$.  Then $\collec{x}\in\mathcal{P}_0$,
and $\collec{x}\in \textup{Cube}\collec{\Delta^*}$ for some
$\collec{\Delta}\in \mathcal{W}$.  Thus for each $i\in[k]$ we have
\[   |\Delta^*_i| \;=\; \frac{k|\Delta_i|}{N}  \;\leq \;
    \frac{k}{N}\left( \frac{N}{k}+\alpha N\right)   \;=\;  1+k\alpha\,.
\]
Since $|x_i|\leq (|\Delta_i^*|+(k-1))\frac{k}{N}$ by Equations 
(\ref{eq.Cubedef}) and (\ref{eq.kcoord}), we obtain
Equation (\ref{eq.YinQ0}).

%Let $Q_1=\left(1+k\alpha+\frac{k^2}{N}\right)Q_0$, which is the 
%right-hand side of Equation (\ref{eq.QYsubsetof}).  

Next we shall show that $Q_1\setminus \mathcal{Y}$ has small 
measure compared to $Q_1$.  
Let $\collec{x}\in Q_1\setminus\mathcal{Y}$, and define the point
$\collec{D}\in\mathbb{Z}^k$ by
\[    D_i \;=\; \left\lfloor \frac{N\,x_i}{k}\right\rfloor 
  \quad (i\in [k-1])
    \hspace{5mm}\hbox{and}\hspace{5mm}  D_k \;=\;
    -\sum_{i=1}^{k-1}D_i \,.
\]
Then $\collec{D}\in\mathcal{P}_0$ and $\collec{x}\in\textup{Cube}\collec{D^*}$.  
Since $\collec{x}\not\in\mathcal{Y}$, the point $\collec{D}$ cannot 
be in $\mathcal{W}$.  This means that one of two inequalities hold:
either
\\
(I)  $|D_i|>n_i-B$  for some $i\in [k]$, or 
\\
(II)  $\left| \frac{D_iN}{n_i}-\frac{D_{j}N}{n_{j}}  \right|
 \,\leq \, 8A+8k/(1-k\alpha)$
for some $i,j\in [k]$ with $i\neq j$.
\\
On the one hand, if (I) holds, then 
\begin{equation*}
  |D_i^*|  \;=\;  \frac{k}{N}\,|D_i|  \;>\;
   \frac{k}{N}\left(\frac{N}{k}-\alpha N -B \right)
   \;=\; 1-k\alpha  -\frac{kB}{N}
   \hspace{5mm}  \hbox{for all }i\in[ k];
\end{equation*}
hence $|x_i| > 1-k\alpha-kB/N -k(k-1)/N$ for all $i$ 
(by Equations (\ref{eq.Cubedef}) and (\ref{eq.kcoord})).
Therefore,
\begin{equation}
    \label{eq.caseIresult}
  \collec{x}\;\in\;  Q_1\setminus 
  \left(1-k\alpha-\frac{kB}{N}-\frac{k^2}{N}\right)Q_0
  \hspace{5mm}\hbox{in case (I).}
\end{equation} 
(Notice that $1-k\alpha-kB/N-k^2/N>0$, due to our assumptions 
that $\alpha<1/4k$, $B\leq N/2k$, and $N>4k^2$ from the 
beginning of the proof.)
On the other hand, if (II) holds for given $i$ and $j$, 
and (I) does not hold, then
\begin{eqnarray*}
  |D_i^*-D_j^*|  & = & \left| \frac{kD_i}{N}-\frac{kD_j}{N}\right|
  \\
   & \leq &  \left| \frac{kD_i}{N}- \frac{D_i}{n_i}\right|\, +\,
   \left| \frac{D_i}{n_i}-\frac{D_j}{n_j}\right|  +
   \left|\frac{D_j}{n_j}-  \frac{kD_j}{N}\right|
   \\
   & \leq &  
   \frac{k}{N}\,\frac{|D_i|}{n_i}\,\left| n_i-\frac{N}{k}\right|+
   \frac{8}{N}\left(A+\frac{k}{1-k\alpha}\right)  \,+\,
    \frac{k}{N}\,\frac{|D_j|}{n_j}\,\left|\frac{N}{k}-n_j\right|
    \\
    & \leq & R    \hspace{11mm}\hbox{where} \hspace{4mm}  
       R\,=\, 2k\alpha \,+\,\frac{8}{N}\left(A+\frac{k}{1-k\alpha}\right) \,,
\end{eqnarray*}
and hence $|x_i-x_j|\, \leq \, R+\frac{k}{N}+\frac{k(k-1)}{N}$
(again, using Equations (\ref{eq.Cubedef}) and (\ref{eq.kcoord})).

Summarizing the results of the preceding paragraph, we have 
shown 
\begin{equation}
    \label{eq.QYsubsetof}
    Q_1\setminus \mathcal{Y} \; \;\subset \;\;  
    \left[ Q_1\setminus  \left(1-k\alpha-\frac{kB}{N} 
    -\frac{k^2}{N}\right)Q_0  \right]
     \;\; \cup \;  \bigcup_{1\leq i<j\leq k} 
     \mathcal{D}_{ij}\left(R+\frac{k^2}{N}\right)
\end{equation}
where we define
\[  \mathcal{D}_{ij}(r) \;:=\;  \{\collec{x}\in\mathcal{P}_0 \cap 
  [-1,1]^k\,:\,  |x_i-x_j|\, \, \leq r \} \,.
\]
Write Leb$_{k-1}$ for $(k-1)$-dimensional Lebesgue measure.
Then we have
\begin{equation}
    \label{eq.lebscalet}
      \textup{Leb}_{k-1}(tQ_0)  \;=\;  t^{k-1}\textup{Leb}_{k-1}(Q_0)
      \hspace{5mm} \hbox{ for any $t>0$}, 
\end{equation}
from which it follows that
\begin{equation}
    \label{eq.lebscale}
        \textup{Leb}_{k-1}\left( Q_1\setminus  
        \left(1-k\alpha-\frac{kB}{N}-\frac{k^2}{N}\right)Q_0 \right)
        \;=\;
        \left[1- \left(\frac{1-k\alpha-\frac{kB}{N}-\frac{k^2}{N} }{
         1+ k\alpha+ \frac{k^2}{N} }\right)^{k-1}\right]
        \textup{Leb}_{k-1}(Q_1) \,.
\end{equation}
Next, we make three observations for $i,j\in[k]$ with $i\neq j$.
\\
(\textit{a}) A set of the form $\mathcal{D}_{ij}(r)$ lies
between two parallel hyperplanes $x_i-x_j=\pm r$, which are distance
$\sqrt{2}\,r$ apart.  
\\ 
(\textit{b})  The normal vector to any hyperplane 
$x_i-x_j=$ \textit{Constant} is perpendicular to the normal vector of 
$\mathcal{P}_0$; and
\\
(\textit{c}) The diameter of $[-1,1]^k$ is $2\sqrt{k}$. 
\\
By observation (\textit{b}), we can choose an orthonormal 
basis $\{\collec{e^{(\ell)}}:\ell\in[k]\}$ such that
$\collec{e^{(1)}}$ is orthogonal to hyperplanes 
$x_i-x_j=$ \textit{Constant} and
$\collec{e^{(k)}}$ is orthogonal to $\mathcal{P}_0$.
Let $\textbf{H}$ be the set of all vectors in $\R^k$ of the form 
$\sum_{\ell=1}^{k-1}t_{\ell}\collec{e^{(\ell)}}$ such that 
$|t_{\ell}|\leq \sqrt{k}$ for every $\ell\in [k{-}1]$.
Then $\textbf{H}$ is a $(k{-}1)$-dimensional 
hypercube of side length
$2\sqrt{k}$ centered at the origin, contained in $\mathcal{P}_0$,
with two of its faces contained in the two hyperplanes 
$x_i-x_j=\pm \sqrt{2k}$.
By (\textit{c}), this hypercube $\textbf{H}$ contains $Q_0$.
By (\textit{a}), Leb$_{k-1}(\mathcal{D}_{ij}(r))\,\leq \, 
\sqrt{2}\,r\times (2\sqrt{k})^{k-2}$.
Inserting this and Equation (\ref{eq.lebscale}) into Equation 
(\ref{eq.QYsubsetof}) yields
\begin{eqnarray}
    \nonumber 
    \textup{Leb}_{k-1}(Q_1\setminus \mathcal{Y})
        &\leq &
     % \left(1-k\alpha-\frac{k^2}{N}\right)^{k-1} 
       \left[1- \left(\frac{1-k\alpha-\frac{kB}{N}-\frac{k^2}{N} }{
         1+ k\alpha+ \frac{k^2}{N} }\right)^{k-1}\right]
      \textup{Leb}_{k-1}(Q_1) 
      \\
       \label{eq.QYprob}   
    & & \hspace{4mm}  +\,\;
   \binom{k}{2} \sqrt{2}\left( 2k\alpha+
      \frac{1}{N}\left(8A+k^2+\frac{8k}{1-k\alpha}\right)
      \right) (2\sqrt{k})^{k-2}
\end{eqnarray}

Now we can put the pieces together.  
\begin{align*}
   \wass_2\left(\lambda^{\widehat{\mathcal{W}}}\,,\,
  \lambda^{Q_0}\right)  & \leq
    \wass_1\left(\textup{Unif}(\mathcal{W})\,,\,
     \textup{Unif}(Q_0)\right) 
     \hspace{16mm}\hbox{(by Lemma \ref{lem.lambdaQ})}
     \\
   & \leq 
   \wass_1\left(\textup{Unif}(\mathcal{\widehat{W}})\,,\,
     \textup{Unif}(\mathcal{W^*}) \right) \,+\,
     \wass_1\left( \textup{Unif}(\mathcal{W^*}) \,,\,
        \textup{Unif}(\mathcal{Y}) \right) 
        \\
      & \quad  \,+\,
       \wass_1\left( \textup{Unif}(\mathcal{Y})\,,\,
         \textup{Unif}(Q_1)\right) \,+\,
       \wass_1\left( \textup{Unif}(Q_1)\,,\,
         \textup{Unif}(Q_0)\right)  
         \\
         & \leq k^2\alpha \,+\, \frac{2k^2}{N}  
      \, +\,\textup{diam}(Q_1)\left[
          \frac{ \textup{Leb}_{k-1}(Q_1\setminus\mathcal{Y})}{
          \textup{Leb}_{k-1}(Q_1)} \,+\,
         \frac{ \textup{Leb}_{k-1}(Q_1\setminus Q_0)}{
          \textup{Leb}_{k-1}(Q_1)}  \right]
          \\
          & \hspace{5mm}\hbox{(by Equations (\ref{eq.wassbdWW})
          and (\ref{eq.wassUYUW}), and Lemma \ref{lem.unif})}
          \\
          & = O\left(\alpha +\frac{A+B+1}{N}\right) 
          \\
        & \hspace{5mm}\hbox{(by Equations (\ref{eq.QYprob}),
         (\ref{eq.lebscalet}), and (\ref{eq.YinQ0}) ).} \qedhere
\end{align*}
\end{proof}

\begin{prop}
   \label{prop.wassSIm}
Given $\alpha\in (0,1/(3k))$, 
let $\Im(N,\alpha)$ be the image     
%under $\Psi$ 
of $\Dom(N,\alpha, \alpha N,2\alpha N)$ under $\Psi$.  
\\
(a) For sufficiently large $N$, we have $\Im(N,\alpha)\subset \E_N(\deck)$ and the restriction of $\Psi$ to
$\Dom(N,\alpha, \alpha N,2\alpha N)$ is $k!$-to-one.
\\
(b) Moreover, for any $\epsilon>0$, there exists an 
$\alpha_{\epsilon}\in (0,1/4k)$ such that
\begin{equation} 
  \label{eq.wasseps}
  \limsup_{N\to\infty}\wass_2\left(\hat{\mu}[\E_N(\deck)] \,,\,
  \hat{\mu}[\Im(N,\alpha)]\right)  \;<\; \epsilon
   \hspace{5mm}\hbox{whenever $0<\alpha<\alpha_{\epsilon}$}.
\end{equation}
\end{prop}

\begin{proof}
Part (a) follows from Corollary \ref{cor.Dombound} and Lemma
\ref{lem.V**strips}.   Next, by Lemma \ref{lem.unif}, 
\[   \wass_2\left(\hat{\mu}[\E_N(\deck)] \,,\,
  \hat{\mu}[\Im(N,\alpha)]\right) \;\leq \;
    \textup{diam}(\pargram) \,\left(1- \frac{|\Im(N,\alpha)|
    }{| \E_N(\deck) |}    \right) \,.
\]
Part (b) follows from part (a), Equation (\ref{eq.SNDomsqueeze}), 
and Theorem \ref{thm.E4321}.
\end{proof}

\begin{proof}[Proof of Theorem \ref{thm.wassmain}]
In this proof, we shall let $A=\alpha N$ and $B=2\alpha N$, where 
$\alpha$ is a small positive constant in $(0,1/4k)$.

%(Follows from Propositions \ref{prop.domwass1}, \ref{prop.wassnn},
%and \ref{prop.wassSIm}.)
Let $\collec{n}\in \mathfrak{N}(N,\alpha)$ and 
$\collec{\Delta}\in \mathcal{W}(\alpha,A,B,\collec{n})$.
Then for every 
$\collec{G},\collec{H}\in \mathcal{V}_N^{**A}(\collec{n})$, 
Proposition \ref{prop.domwass1} tells us that
\[   \wass_1\left(\lambda\collec{\Delta/n}, \, 
\hat{\mu}_{\Psi(\collec{n},\collec{G},\collec{H},\collec{\Delta})}
 \right) \;\leq \; 
 \frac{1}{N}\left(2A+\frac{4k}{1-k\alpha}\right) \;+\;4k\alpha \,,
\]
and hence that
\begin{equation}
  \label{eq.deltalambda}
 \wass_2\left(\delta_{\lambda\collec{\Delta/n}}, \, \delta_{
\hat{\mu}_{\Psi(\collec{n},\collec{G},\collec{H},\collec{\Delta})}}
  \right)   
  \;\leq \;   \frac{1}{N}\left(2A+\frac{4k}{1-k\alpha}\right)
    \;+\; 4k\alpha \,.
\end{equation}
Next, for such $\collec{n}$ and $\collec{\Delta}$, define the mixture
\[
   \mathcal{M}(\collec{n},\collec{\Delta}) \; :=\;
   \frac{1}{|\mathcal{V}_N^{**A}(\collec{n})|^2}  
   \sum_{\collec{G},\collec{H}\in \mathcal{V}_N^{**A}(\collec{n}) }
    \delta_{
\hat{\mu}_{\Psi(\collec{n},\collec{G},\collec{H},\collec{\Delta})}}
 \,.
\]
It follows from Lemma \ref{lem.wass1} and Equation (\ref{eq.deltalambda}) that
\begin{equation}
    \label{eq.wassMlambda}
    \wass_2\left( \delta_{\lambda\collec{\Delta/n}}, \,
      \mathcal{M}(\collec{n},\collec{\Delta}) \right)
       \;\leq \;   \frac{1}{N}\left(2A+\frac{4k}{1-k\alpha}\right)
       \;+\; 4k\alpha \,.
\end{equation}

By the properties of $\Psi$ described in Proposition \ref{prop.wassSIm}, we obtain  (writing $\Dom$ for
$\Dom(N,\alpha,\alpha N, 2\alpha N)$ and 
$\mathcal{W}\collec{n}$ for $\mathcal{W}(\collec{n},\alpha,\alpha N,
2\alpha N)$)
\begin{eqnarray}
    \nonumber
    \hat{\mu}[\Im(N,\alpha)]  & = & 
    \frac{1}{|\Im(N,\alpha)|}
    \sum_{\sigma\in \mathbf{Im}(N,\alpha)} 
    \delta_{\hat{\mu}_{\sigma}}  \\
    \nonumber
    & = &  \frac{1}{k! \,|\Im(N,\alpha)|}
    \sum_{\vec{v}\in \mathbf{Dom}} 
    \delta_{\hat{\mu}_{\Psi(\vec{v})}}  \\ 
      \label{eq.imnasum}
    & = & \frac{1}{|\Dom|} 
    \sum_{\collec{n}\in \mathfrak{N}(N,\alpha)}
    \sum_{\collec{\Delta}\in \mathcal{W}\collec{n}} 
    |\mathcal{V}_N^{**A}(\collec{n})|^2 \,
      \mathcal{M}(\collec{n},\collec{\Delta}) \,.
\end{eqnarray}

Define the random measure 
\[   \lambda^*  \;=\;  \sum_{\collec{n}\in \mathfrak{N}(N,\alpha)}
     \frac{|\mathcal{W}\collec{n}|\, |\mathcal{V}_N^{**A}(\collec{n})|^2}{|\Dom|} \,
     \lambda^{\widehat{\mathcal{W}\collec{n}}} 
\]
(c.f.\ Equation (\ref{eq.W**0})).  Then we can write
\begin{equation}
    \label{eq.lamstarsum}
     \lambda^*  \;=\;   \frac{1}{|\Dom|} 
    \sum_{\collec{n}\in \mathfrak{N}(N,\alpha)}
    \sum_{\collec{\Delta}\in \mathcal{W}\collec{n}} 
    |\mathcal{V}_N^{**A}(\collec{n})|^2 \,
      \delta_{\lambda\collec{n/\Delta}} \,.
\end{equation}
Then by Equations (\ref{eq.wassMlambda}--\ref{eq.lamstarsum})
and Lemma \ref{lem.wass1}, we obtain
\begin{equation}
    \label{eq.wassmulamstar}
    \wass_2\left(\lambda^*,\hat{\mu}[\Im(N,\alpha)]\right)
       \;\leq \; \frac{1}{N}\left(2A+\frac{4k}{1-k\alpha}\right)
       \;+\; 4k\alpha \,.
\end{equation}

By Proposition \ref{prop.wassnn} and Lemma \ref{lem.wass1}, we have
\begin{equation}
    \label{eq.lamstarlamQ}
    \wass_2\left(\lambda^*,\lambda^{Q_0}\right)
       \;\leq \; C\left(\alpha+\frac{A+B+1}{N}\right) \,.
\end{equation}
The triangle inequality gives
\begin{align}
   \nonumber
   \wass_2 & \left(\hat{\mu}[\E_N(\deck)] \,,\,\lambda^{Q_0} \right)
    \\
    \nonumber
    & \leq \;
     \wass_2\left(\hat{\mu}[\E_N(\deck)] \,,\, 
      \hat{\mu}[\Im(N,\alpha)]\right)  
      \\
      \label{eq.wassfintri}
      & \hspace{6mm}  \,+\,
     \wass_2\left( \hat{\mu}[\Im(N,\alpha)]\,,\,
       \lambda^* \right)  \,+\,
        \wass_2\left(\lambda^*,\lambda^{Q_0}\right) \,.
\end{align}
Let $\epsilon>0$ and let $\alpha_{\epsilon}\in (0,1/4k)$ be as 
specified in Equation (\ref{eq.wasseps}).  
Let $\alpha\in (0,\alpha_{\epsilon})$.  
Applying Equations (\ref{eq.wasseps}), (\ref{eq.wassmulamstar}), and
(\ref{eq.lamstarlamQ}) to Equation (\ref{eq.wassfintri}), we have
\[
    \limsup_{N\to\infty}  \wass_2 
    \left(\hat{\mu}[\E_N(\deck)] \,,\,\lambda^{Q_0} \right)
    \;\leq \;  \epsilon \,+\,  2\alpha \,+\, 4k\alpha \,+ 
    \, 4C\alpha \,.
\]
Since $\epsilon$ and $\alpha$ can both be chosen to be arbitrarily 
small, the above limsup must be zero.  This proves the theorem.
\end{proof}

\section*{Acknowledgments}

We are grateful to Tom Salisbury for a helpful discussion about
random measures.

\end{document}

%% file: Fig321.tex
\setlength{\unitlength}{1.2mm}
\begin{figure}
  \begin{center}
%\begin{picture}(140,150)
%
%\put(40,90){
\begin{picture}(40,70)(-5,-25)
%\put(-23,0){\vector(1,0){68}}
%\put(0,-23){\vector(0,1){68}}
\put(2,2){\line(1,0){20}}
\put(2,2){\line(0,1){20}}
\put(2,22){\line(1,0){20}}
\put(22,2){\line(0,1){20}}
\put(2,0){\line(0,1){1}}
\put(1,-3){$1$}
\put(22,0){\line(0,1){1}}
\put(20,-3){$N$}
\put(0,2){\line(1,0){1}}
\put(-3,1){$1$}
\put(0,22){\line(1,0){1}}
\put(-3,21){$N$}
\put(0,15){\line(1,0){1}}
%   \put(-6,14){$+\Delta$}
\put(0,-11){\line(1,0){1}}
%    \put(-6,-12){$-\Delta$}
%\put(13,8){\huge{${\sigma}$}}
\multiput(2,15)(2,2){11}{\circle*{0.6}}
\multiput(2,-11)(2,2){11}{\circle*{0.6}}
%\put(22,16){\circle*{0.8}}
%
%\paws
\put(2,22){\textcolor{blue}{\line(1,1){20}}}
\put(2,-18){\textcolor{blue}{\line(1,1){20}}}
\put(2,2){\textcolor{blue}{\line(1,1){20}}}
\put(2,14){\textcolor{red}{\line(1,1){20}}}
\put(2,16){\textcolor{red}{\line(1,1){20}}}
\put(2,-10){\textcolor{red}{\line(1,1){20}}}
\put(2,-12){\textcolor{red}{\line(1,1){20}}}
%\multiput(-20,-20)(4,4){17}{\textcolor{blue}{$\cdot$}}       %{\line(1,1){1}}
%
%\put(30,11){$N{-}1$}  
%\put(33,14){\vector(0,1){8}}
%\put(33,10){\vector(0,-1){8}}
%\put(30,31){$N{-}1$}  
%\put(33,34){\vector(0,1){8}}
%\put(33,30){\vector(0,-1){8}}
%
\end{picture}
  \end{center}
  \caption{Sketch of typical 321-avoiding bounded affine permutation
   of size $N$.   
   The plot is completely covered by two diagonal strips of width
   $\alpha N$ where $\alpha$ is a small positive number.
   For $(m(m{-}1)\cdots 21)$-avoidance, we would need $m-1$ such strips.
   In a typical plot, each strip covers an approximately equal
   number of points.}
  \label{fig.321}
\end{figure}

%% file: Fig321limit.tex
\setlength{\unitlength}{1mm}
\begin{figure}
  \begin{center}
\begin{picture}(100,70)
\put(0,0){
\begin{picture}(30,70)(-5,-25)
%\put(-23,0){\vector(1,0){68}}
%\put(0,-23){\vector(0,1){68}}
\put(2,2){\line(1,0){20}}
\put(2,2){\line(0,1){20}}
\put(2,22){\line(1,0){20}}
\put(22,2){\line(0,1){20}}
\put(2,0){\line(0,1){1}}
\put(1,-3){$1$}
\put(22,0){\line(0,1){1}}
\put(20,-3){$N$}
\put(0,2){\line(1,0){1}}
\put(-3,1){$1$}
\put(0,22){\line(1,0){1}}
\put(-3,21){$N$}
\put(0,15){\line(1,0){1}}
% \put(-6,14){$+\Delta$}
\put(0,-11){\line(1,0){1}}
% \put(-6,-12){$-\Delta$}
%\put(13,8){\huge{${\sigma}$}}
\multiput(2,15)(2,2){11}{\circle*{0.6}}
\multiput(2,-11)(2,2){11}{\circle*{0.6}}
%\put(22,16){\circle*{0.8}}
%
%\paws
\put(2,22){\textcolor{blue}{\line(1,1){20}}}
\put(2,-18){\textcolor{blue}{\line(1,1){20}}}
\put(2,2){\textcolor{blue}{\line(1,1){20}}}
\put(2,14){\textcolor{red}{\line(1,1){20}}}
\put(2,16){\textcolor{red}{\line(1,1){20}}}
\put(2,-10){\textcolor{red}{\line(1,1){20}}}
\put(2,-12){\textcolor{red}{\line(1,1){20}}}
%\multiput(-20,-20)(4,4){17}{\textcolor{blue}{$\cdot$}}       %{\line(1,1){1}}
%
%\put(30,11){$N{-}1$}  
%\put(33,14){\vector(0,1){8}}
%\put(33,10){\vector(0,-1){8}}
%\put(30,31){$N{-}1$}  
%\put(33,34){\vector(0,1){8}}
%\put(33,30){\vector(0,-1){8}}
%
\end{picture}
}

\put(45,34){$\Longrightarrow$}

\put(70,0){
\begin{picture}(30,70)(-5,-25)
%\put(-23,0){\vector(1,0){68}}
%\put(0,-23){\vector(0,1){68}}
\put(2,2){\line(1,0){20}}
\put(2,2){\line(0,1){20}}
\put(2,22){\line(1,0){20}}
\put(22,2){\line(0,1){20}}
\put(2,0){\line(0,1){1}}
\put(1,-3){$0$}
\put(22,0){\line(0,1){1}}
\put(20,-3){$1$}
\put(0,2){\line(1,0){1}}
\put(-3,1){$0$}
\put(0,22){\line(1,0){1}}
\put(-3,21){$1$}
\put(0,15){\line(1,0){1}}
\put(-6,14){$+\delta$}
\put(0,-11){\line(1,0){1}}
\put(-6,-12){$-\delta$}
%\put(13,8){\huge{${\sigma}$}}
%\put(22,16){\circle*{0.8}}
%
%\paws
\put(2,22){\textcolor{blue}{\line(1,1){20}}}
\put(2,-18){\textcolor{blue}{\line(1,1){20}}}
%\put(2,2){\textcolor{blue}{\line(1,1){20}}}
\put(2,15){\textcolor{red}{\line(1,1){20}}}
\put(2,-11){\textcolor{red}{\line(1,1){20}}}
\end{picture}
}
\end{picture}
 \end{center}
\caption{Scaling of a random element of $\E_N(321)$ 
to a permuton-like limit.  As $N\rightarrow\infty$, we rescale each 
axis of $\mathbb{R}^2$ by $1/N$.  We view the left figure as a 
discrete probability measure with an atom of mass $1/N$ at each point
of the plot. 
The right figure represents the uniform probability measure on the 
line segments $y=x+\delta$ and $y=x-\delta$ ($0\leq x\leq 1$),
where $\delta$ is uniformly distributed on $[-1,1]$.  }
  \label{fig.321limit}
\end{figure}

%% file: FigABAB.tex
\begin{figure}
%    \setbeamercovered{invisible}
\centering
\begin{tikzpicture}[scale=0.5]
\draw[fill=blue!30] (1,0) rectangle (4,1);
\draw[fill=blue!30] (1,3) rectangle (4,5);
\draw[fill=blue!30] (1,6) rectangle (4,9);
\draw[fill=blue!30] (6,0) rectangle (8,1);
\draw[fill=blue!30] (6,3) rectangle (8,5);
\draw[fill=blue!30] (6,6) rectangle (8,9);
\draw[fill=blue!30] (9,0) rectangle (10,1);
\draw[fill=blue!30] (9,3) rectangle (10,5);
\draw[fill=blue!30] (9,6) rectangle (10,9);
\draw[fill=green!52] (0,1) rectangle (1,3);
\draw[fill=green!52] (0,5) rectangle (1,6);
\draw[fill=green!52] (0,9) rectangle (1,10);
\draw[fill=green!52] (4,1) rectangle (6,3);
\draw[fill=green!52] (4,5) rectangle (6,6);
\draw[fill=green!52] (4,9) rectangle (6,10);
\draw[fill=green!52] (8,1) rectangle (9,3);
\draw[fill=green!52] (8,5) rectangle (9,6);
\draw[fill=green!52] (8,9) rectangle (9,10);
\draw[blue] (0,0) grid (10,10);
\draw[very thick, scale=10] (0,0) grid (1,1);
\node[teal] at (0.5,-0.5){A};
\node[blue] at (1.5,-0.5){B};
\node[blue] at (2.5,-0.5){B};
\node[blue] at (3.5,-0.5){B};
\node[teal] at (4.5,-0.5){A};
\node[teal] at (5.5,-0.5){A};
\node[blue] at (6.5,-0.5){B};
\node[blue] at (7.5,-0.5){B};
\node[teal] at (8.5,-0.5){A};
\node[blue] at (9.5,-0.5){B};
\node[blue] at (-0.5,0.5){B};
\node[teal] at (-0.5,1.5){A};
\node[teal] at (-0.5,2.5){A};
\node[blue] at (-0.5,3.5){B};
\node[blue] at (-0.5,4.5){B};
\node[teal] at (-0.5,5.5){A};
\node[blue] at (-0.5,6.5){B};
\node[blue] at (-0.5,7.5){B};
\node[blue] at (-0.5,8.5){B};
\node[teal] at (-0.5,9.5){A};
% \paws
%
\draw[fill] (0.5,1.5) circle [radius=0.1];
\draw[fill] (4.5,2.5) circle [radius=0.1];
\draw[fill] (5.5,5.5) circle [radius=0.1];
\draw[fill] (8.5,9.5) circle [radius=0.1];
\draw[fill,red] (1.5,0.5) circle [radius=0.1];
\draw[fill,red] (2.5,3.5) circle [radius=0.1];
\draw[fill,red] (3.5,4.5) circle [radius=0.1];
\draw[fill,red] (6.5,6.5) circle [radius=0.1];
\draw[fill,red] (7.5,7.5) circle [radius=0.1];
\draw[fill,red] (9.5,8.5) circle [radius=0.1];

\end{tikzpicture}

\caption{The case $k=2$:  Constructing a \textbf{321}-avoiding 
ordinary permutation of size 10 from Equation (\ref{eq.ghord}).  
Here $n_1=4$ and $n_2=6$.
Each A on the horizontal (respectively, vertical) axis indicates a 
member of $G_1$ (respectively, $H_1$), while the B's indicate members
of $G_2$ (and $H_2$).  The dots are the points of the
plot, as we pair the $j^{th}$ element of $G_i$ with the $j^{th}$ 
element of $H_i$.  The shaded squares (green for $G_1\times H_1$
and blue for $G_2\times H_2$) will be needed when 
we extend this construction to affine permutations in Figure 
\ref{fig.ABABaffine}.}
   \label{fig.ABAB}

\end{figure}
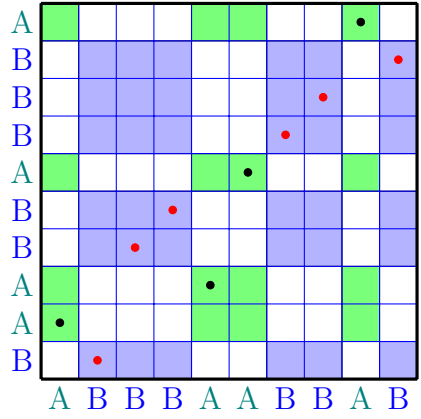

%% file: FigABABaffine.tex
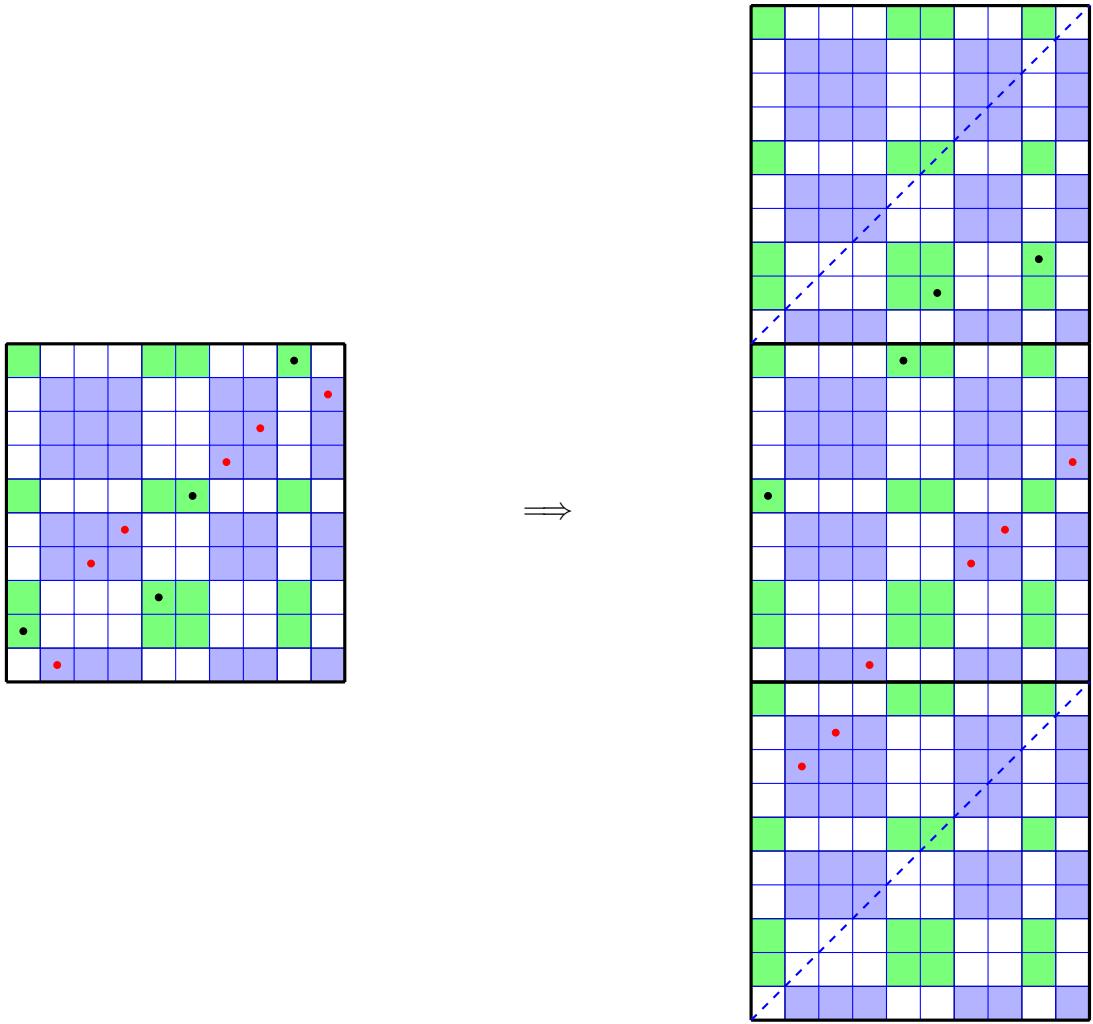
\begin{figure}
%    \setbeamercovered{invisible}
\begin{center}   %\centering
\begin{tikzpicture}[scale=0.45]       %=0.27]

\draw[fill=blue!30] (1,0) rectangle (4,1);
\draw[fill=blue!30] (1,3) rectangle (4,5);
\draw[fill=blue!30] (1,6) rectangle (4,9);
\draw[fill=blue!30] (6,0) rectangle (8,1);
\draw[fill=blue!30] (6,3) rectangle (8,5);
\draw[fill=blue!30] (6,6) rectangle (8,9);
\draw[fill=blue!30] (9,0) rectangle (10,1);
\draw[fill=blue!30] (9,3) rectangle (10,5);
\draw[fill=blue!30] (9,6) rectangle (10,9);
\draw[fill=green!52] (0,1) rectangle (1,3);
\draw[fill=green!52] (0,5) rectangle (1,6);
\draw[fill=green!52] (0,9) rectangle (1,10);
\draw[fill=green!52] (4,1) rectangle (6,3);
\draw[fill=green!52] (4,5) rectangle (6,6);
\draw[fill=green!52] (4,9) rectangle (6,10);
\draw[fill=green!52] (8,1) rectangle (9,3);
\draw[fill=green!52] (8,5) rectangle (9,6);
\draw[fill=green!52] (8,9) rectangle (9,10);
\draw[blue] (0,0) grid (10,10);
\draw[very thick, scale=10] (0,0) grid (1,1);
\draw[fill] (0.5,1.5) circle [radius=0.1];
\draw[fill] (4.5,2.5) circle [radius=0.1];
\draw[fill] (5.5,5.5) circle [radius=0.1];
\draw[fill] (8.5,9.5) circle [radius=0.1];
\draw[fill,red] (1.5,0.5) circle [radius=0.1];
\draw[fill,red] (2.5,3.5) circle [radius=0.1];
\draw[fill,red] (3.5,4.5) circle [radius=0.1];
\draw[fill,red] (6.5,6.5) circle [radius=0.1];
\draw[fill,red] (7.5,7.5) circle [radius=0.1];
\draw[fill,red] (9.5,8.5) circle [radius=0.1];

\node at (16,5){$\Longrightarrow$};

\begin{scope}[shift={(22,-10)}]

\draw[fill=blue!30] (1,0) rectangle (4,1);
\draw[fill=blue!30] (1,3) rectangle (4,5);
\draw[fill=blue!30] (1,6) rectangle (4,9);
\draw[fill=blue!30] (6,0) rectangle (8,1);
\draw[fill=blue!30] (6,3) rectangle (8,5);
\draw[fill=blue!30] (6,6) rectangle (8,9);
\draw[fill=blue!30] (9,0) rectangle (10,1);
\draw[fill=blue!30] (9,3) rectangle (10,5);
\draw[fill=blue!30] (9,6) rectangle (10,9);
\draw[fill=green!52] (0,1) rectangle (1,3);
\draw[fill=green!52] (0,5) rectangle (1,6);
\draw[fill=green!52] (0,9) rectangle (1,10);
\draw[fill=green!52] (4,1) rectangle (6,3);
\draw[fill=green!52] (4,5) rectangle (6,6);
\draw[fill=green!52] (4,9) rectangle (6,10);
\draw[fill=green!52] (8,1) rectangle (9,3);
\draw[fill=green!52] (8,5) rectangle (9,6);
\draw[fill=green!52] (8,9) rectangle (9,10);
\draw[blue] (0,0) grid (10,10);
\draw[very thick, scale=10] (0,0) grid (1,1);
\draw[fill=blue!30] (1,10) rectangle (4,11);
\draw[fill=blue!30] (1,13) rectangle (4,15);
\draw[fill=blue!30] (1,16) rectangle (4,19);
\draw[fill=blue!30] (6,10) rectangle (8,11);
\draw[fill=blue!30] (6,13) rectangle (8,15);
\draw[fill=blue!30] (6,16) rectangle (8,19);
\draw[fill=blue!30] (9,10) rectangle (10,11);
\draw[fill=blue!30] (9,13) rectangle (10,15);
\draw[fill=blue!30] (9,16) rectangle (10,19);
\draw[fill=green!52] (0,11) rectangle (1,13);
\draw[fill=green!52] (0,15) rectangle (1,16);
\draw[fill=green!52] (0,19) rectangle (1,20);
\draw[fill=green!52] (4,11) rectangle (6,13);
\draw[fill=green!52] (4,15) rectangle (6,16);
\draw[fill=green!52] (4,19) rectangle (6,20);
\draw[fill=green!52] (8,11) rectangle (9,13);
\draw[fill=green!52] (8,15) rectangle (9,16);
\draw[fill=green!52] (8,19) rectangle (9,20);
\draw[blue] (0,10) grid (10,20);
\draw[very thick, scale=10] (0,1) grid (1,2);
\draw[fill=blue!30] (1,20) rectangle (4,21);
\draw[fill=blue!30] (1,23) rectangle (4,25);
\draw[fill=blue!30] (1,26) rectangle (4,29);
\draw[fill=blue!30] (6,20) rectangle (8,21);
\draw[fill=blue!30] (6,23) rectangle (8,25);
\draw[fill=blue!30] (6,26) rectangle (8,29);
\draw[fill=blue!30] (9,20) rectangle (10,21);
\draw[fill=blue!30] (9,23) rectangle (10,25);
\draw[fill=blue!30] (9,26) rectangle (10,29);
\draw[fill=green!52] (0,21) rectangle (1,23);
\draw[fill=green!52] (0,25) rectangle (1,26);
\draw[fill=green!52] (0,29) rectangle (1,30);
\draw[fill=green!52] (4,21) rectangle (6,23);
\draw[fill=green!52] (4,25) rectangle (6,26);
\draw[fill=green!52] (4,29) rectangle (6,30);
\draw[fill=green!52] (8,21) rectangle (9,23);
\draw[fill=green!52] (8,25) rectangle (9,26);
\draw[fill=green!52] (8,29) rectangle (9,30);
\draw[blue] (0,20) grid (10,30);
\draw[very thick, scale=10] (0,2) grid (1,3);
\draw[dashed,thick,blue] (0,0)--(10,10) (0,20)--(10,30);

%   \paws
\draw[fill] (0.5,15.5) circle [radius=0.1];   %0.25];
% \node  at (-15,23){pick $\Delta\in\mathbb{Z}$; e.g.\ $\Delta=+2$};
% \node[black] at (-15,7){Shift black vertically $\Delta$ green squares};
%  \paws
\draw[fill] (4.5,19.5) circle [radius=0.1];
\draw[fill] (5.5,21.5) circle [radius=0.1];
\draw[fill] (8.5,22.5) circle [radius=0.1];
%
% \node[red] at (-15,4){Shift red vertically $-\Delta$ blue squares};
\draw[fill,red] (1.5,7.5) circle [radius=0.1];   %  0.2];
\draw[fill,red] (2.5,8.5) circle [radius=0.1];
\draw[fill,red] (3.5,10.5) circle [radius=0.1];
\draw[fill,red] (6.5,13.5) circle [radius=0.1];
\draw[fill,red] (7.5,14.5) circle [radius=0.1];
\draw[fill,red] (9.5,16.5) circle [radius=0.1];

% \paws
% \node at (-15,1) {May exceed bounds, but still centered!};

\end{scope}

\end{tikzpicture}

\end{center}
\caption{The action of $\Psi$ for $k=2$ and $N=10$:
The grid on the left is the example from Figure \ref{fig.ABAB}.
On the right, a copy of the same grid colouring (without the dots) 
is placed above and below the original grid, extending the original 
colouring
periodically from $\{1,\ldots,N\}^2$ to $\{1,\ldots,N\}\times
\{-N{+}1,\ldots,2N\}$. The boundedness condition requires that 
all points of the plot lie strictly between the two diagonal dashed 
lines shown.  The plotted points on the right illustrate the
result of taking $\Delta_1=2$ and $\Delta_2=-2$ in Equation (\ref{eq.sigma4def}).  
This has the effect that each black (respectively, red) dot moves up
two green spaces (respectively, down two blue spaces) 
from its original position (i.e.\ in the left grid).}
  \label{fig.ABABaffine}
\end{figure}